\documentclass[12pt,reqno]{amsart}
\usepackage{geometry}                
\usepackage{amsmath,amssymb}
\usepackage{graphicx}
\usepackage{amssymb,latexsym, amsmath}
\usepackage{amsfonts,amsthm}
\usepackage{amscd}
\usepackage{mathrsfs}
\usepackage{hyperref}
\usepackage{eucal}
\usepackage{epstopdf}
\usepackage{amsgen}
\usepackage{xspace}
\usepackage{verbatim}
\usepackage{stmaryrd}
\usepackage{enumitem}
\newlist{steps}{enumerate}{1}
\setlist[steps, 1]{label = Step \arabic*:}

\newcommand{\gd}{\Delta}

\newcommand{\tu}{\Tilde{u}}

\newcommand{\tg}{\Tilde{g}}
\newcommand{\tJ}{\Tilde{J}}
\newcommand{\tap}{\Tilde{\alpha}}
\newcommand{\tgb}{\Tilde{\beta}}
\newcommand{\tgl}{\Tilde{\lambda}}
\newcommand{\tq}{\Tilde{q}}

\newcommand{\tL}{\Tilde{L}}
\newcommand{\txi}{\Tilde{\xi}}

\newcommand{\tV}{\Tilde{V}}
\newcommand{\tf}{\Tilde{f}}
\newcommand{\Th}{\Tilde{h}}

\newcommand{\inpt}[1]{\langle #1 \rangle}

\newcommand{\gw}{\Omega}
\newcommand{\ap}{\alpha}
\newcommand{\ga}{\gamma}
\newcommand{\de}{\delta}
\newcommand{\gb}{\beta}
\newcommand{\G}{\Gamma}
\newcommand{\gl}{\lambda}
\newcommand{\gL}{\Lambda}
\newcommand{\gs}{\sigma}

\newcommand{\ms}{\mathscr}

\newcommand{\nb}{\nabla}

\newcommand{\ve}{\varepsilon}
\newcommand{\pdr}{\partial}

\newcommand{\beq}{\begin{equation}}
\newcommand{\eeq}{\end{equation}}
\newcommand{\bea}{\begin{align}}
\newcommand{\eea}{\end{align}}
\newcommand{\bthm}{\begin{theorem}}
\newcommand{\ethm}{\end{theorem}}
\newcommand{\bpr}{\begin{proof}}
\newcommand{\epr}{\end{proof}}
\newcommand{\bcl}{\begin{corollary}}
\newcommand{\ecl}{\end{corollary}}
\newcommand{\bpn}{\begin{proposition}}
\newcommand{\epn}{\end{proposition}}
\newcommand{\bre}{\begin{remark}}
\newcommand{\ere}{\end{remark}}
\newcommand{\bdf}{\begin{definition}}
\newcommand{\edf}{\end{definition}}
\newcommand{\bss}{\begin{align*}}
\newcommand{\ess}{\end{align*}}

\newcommand{\bl}{\label}

\newcommand{\mR}{\mathbb{R}}

\usepackage{subfig}
\newtheorem{theorem}{Theorem}[section]
\newtheorem{corollary}[theorem]{Corollary}

\newtheorem{proposition}[theorem]{Proposition}

\theoremstyle{definition}
\newtheorem{definition}[theorem]{Definition}
\theoremstyle{remark}
\newtheorem{remark}{Remark}

\numberwithin{equation}{section}

\begin{document}

\title[Synchronization of Reaction-Diffusion Neural Networks]{Dynamics and Synchronization of Weakly Coupled Memristive Reaction-Diffusion Neural Networks}

\author[Y. You]{Yuncheng You}
\address{University of South Florida, Tampa, FL 33620, USA}
\email{you@mail.usf.edu}
\thanks{}

\author[J. Tu]{Junyi Tu}
\address{Salisbury University, Salisbury, MD 21801, USA}
\email{jxtu@salisbury.edu}
\thanks{}


\date{\today}



\begin{abstract} 
A new mathematical model of memristive neural networks described by the partly diffusive reaction-diffusion equations with weak synaptic coupling is proposed and investigated. Under rather general conditions it is proved that there exists an absorbing set showing the dissipative dynamics of the solution semiflow in the energy space and multiple ultimate bounds. Through uniform estimates and maneuver of integral inequalities and sharp interpolation inequalities on the interneuron differencing equations, it is rigorously proved that exponential synchronization of the neural network solutions at a uniform convergence rate occurs if the coupling strength satisfies a threshold condition expressed by the system parameters. Applications with numerical simulation to the memristive diffusive Hindmarsh-Rose neural networks and FitzHugh-Nagumo neural networks are also shown.
\end{abstract}

\maketitle
 
\section{Introduction}
In this paper we shall consider a new mathematical model of neural networks described by a system of partly diffusive hybrid differential equations with memristors and weak nonlinear interneuron coupling. 

Let a network of $m$ coupled memristive neuron cells be denoted by $\mathcal{NW} = \{\mathcal{N}_i : i = 1, 2, \cdots, m\}$, where $m \geq 2$ is a positive integer, which is described by the following model of partly diffusive and memristive equations with nonlinear couplings. Each neuron $\mathcal{N}_i, 1 \leq i \leq m$, in the network is presented by the hybrid differential equations:
\beq \bl{NW} 
\begin{split}
	& \frac{\pdr u_i}{\pdr t} = \eta \gd u_i + f (u_i, z_i) - k \tanh (\rho_i) u_i - P u_i \sum_{j=1}^m \G (u_j),  \\
	& \frac{\pdr z_i}{\pdr t}  =  \gL z_i + h(u_i, z_i),    \\[3pt]
	& \frac{\pdr \rho_i}{\pdr t} = a u_i - b \rho_i,
\end{split} 
\eeq
for $t > 0,\; x \in \gw \subset \mathbb{R}^{n}$ ($n \leq 3$), where $\gw$ is a bounded domain with locally Lipschitz continuous boundary $\partial \gw$. Here $\gd$ is the Laplacian operator with respect to the spatial variable $x \in \gw$ and $\gL$ is an $\ell \times \ell$ constant square matrix. In the nonlinear synaptic coupling term for neurons, the function
\beq \bl{Ga}
	\G (s) = \frac{1}{1 + \exp [- r (s - V)]} 
\eeq
is a sigmoidal function. The biological meaning of the parameters in this type of nonlinear weak couplings can be seen in \cite{Co, HR, M, PL, TI}. The coefficient $P > 0$ is the coupling strength. The constant $V \in \mathbb{R}$ is a threshold for neuron bursting, and $r > 0$ shapes the sigmoidal function versus the Heaviside function. 

In this system \eqref{NW}, for $ 1 \leq i \leq m$, the transmembrane electric potential $u_i (t, x)$ and the memductance $\rho_i (t, x)$ of the memristor (which is caused by the electromagnetic induction flux across the neuron membrane) are scalar functions, while  $z_i (t, x)$ can be an $\ell$-dimensional ($\ell \geq 1$) vector function whose components represent various ionic currents in the neuron cell. The memristive-potential coupling $ - k \tanh (\rho_i) u_i $ is a nonlinear term.

We impose the homogeneous Neumann boundary condition on the first component function in \eqref{NW}:
\begin{equation} \label{nbc}
	\frac{\pdr u_i}{\pdr \nu} (t, x) = 0, \quad \text{for} \;\; t > 0,  \; x \in \partial \gw, \quad 1 \leq i \leq m.
\end{equation}
The initial states of the system \eqref{NW} will be denoted by 
\begin{equation} \bl{inc}
	 u_i^0 (x) = u_i(0, x), \; z_i^0 (x) = z_i (0, x) ,  \;\rho_i^0 = \rho_i (0, x), \;\; 1 \leq i \leq m.
\end{equation}
The scalar function $f \in C^1 (\mathbb{R}^{1 + \ell}, \mathbb{R})$ and the vector function $h \in C^1 (\mathbb{R}^{1 + \ell}, \mathbb{R}^{\ell})$ are continuously differentiable and we assume that
\beq \bl{Asp1}
	\begin{split}
	&f(s, \gs) s \leq  - \ap |s|^4 + \gl |s| |\gs| + J, \quad  (s, \gs) \in \mathbb{R}^{1 + \ell}, \\
	&\max \left\{\frac{\pdr f}{\pdr s} (s, \gs), \, \left| \frac{\pdr f}{\pdr \gs} (s, \gs)\right| \right\} \leq \gb,  \; (s, \gs) \in \mathbb{R}^{1 + \ell},
	\end{split}
\eeq
and the $\ell$-dimensional matrix $\gL$ and $h(s, \gs)$ satisfy
\beq \bl{Asp2}
	\begin{split}
	&\langle \gL \gs, \gs \rangle \leq - \ga |\gs|^2, \quad \gs \in \mathbb{R}^\ell,   \\[3pt]
	&h(s, \gs) \gs \leq q |s|^2 |\gs| + L |\gs |,  \quad (s, \gs ) \in \mathbb{R}^{1 + \ell},  \\
	& \left| \frac{\pdr h}{\pdr s}(s, \gs) \right| \leq \xi (|s| + 1), \;\; \frac{\pdr h}{\pdr \gs}(s, \gs) = 0,  \quad (s, \gs ) \in \mathbb{R}^{1 + \ell},
	\end{split}
\eeq
where the parameters $\eta, a, b, k$ in \eqref{NW}, $\ap, \gl, J, \gb$ in \eqref{Asp1}, and $\ga, q, L, \xi $ in \eqref{Asp2} are all positive constants. Note that Assumptions \eqref{Asp1} and \eqref{Asp2} are satisfied by the partly diffusive FitzHugh-Nagumo neural networks \cite{LSY, SkY} and the partly diffusive Hindmarsh-Rose neural networks \cite{Y1, Y2}, which will be shown in Section 6.

In neurobiology, the ionic currents flowing across a neuron cell's membrane cause changes of the membrane potential over time. The resulting electrical signals propagate through the neuron axon and stimulate the dendrites of neighbor neurons by synapses, which constitute a biological neural network. An excitatory neuron's firing consists of successive spiking followed by relatively long period of quiescence called bursting. Chaotic bursting means the number of spikes per burst is irregular.

Synchronization mechanism of various neural networks revealed by mathematical models and analysis is one of the central topics in the research of neuroscience, medical science, and artificial neural networks \cite{AD, B, ET, G, Ha, PRK, VK, WL, Wu}. For many neuron models in terms of ODEs with or without time-delay, analysis of Hopf bifurcations, stability by Lyapunov exponents, the energy or Hamiltonian functions,  and numerical simulations are the main approaches to show synchronization of neuron ensembles or neural networks \cite{ET, Iz, N, PL, TI, ZS}.

Dynamics and synchronization of cellular neural networks and memristive neural networks modeled by partly diffusive Hindmarsh-Rose equations and FitzHugh-Nagumo equations have been studied in the authors' group recently \cite{CPY, PSY, LSY, SkY, Y1, Y2, Y3}. These neural network models of hybrid (PDE-ODE) differential equations reflect the structural feature of neuron cells, which contain the short-branch dendrites receiving incoming signals and the long-branch axon transmitting outgoing signals through synapses. 

Memristor concept coined by Chua \cite{Chua} describes the effect of electromagnetic flux on moving charges such as the ionic currents in neuron cells. In advanced biological neuron models and artificial intelligence computing \cite{Ay, E, Jo, Li, SW, U, VK, WZ}, the memristive feature was exhibited as a new type (other than electrical and chemical) synaptic coupling or an ideal component which has the nonvolatile properties and can process dynamically memorized signal information to deal with complex or chaotic behaviors in neural networks. Memristor-based differential equation models now appear in many fields with applications to image encryption, DNA sequences operation, brain criticality, cell physiology, cybersecurity, drift-diffusion models in semiconductor devices, and quantum computers, cf. \cite{JZ, Li, LH, RK, Sn, SW, W}. 

Researches on dynamics of memristive neural networks in ODE models with linear interneuron couplings are increasing in the recent years, cf.\cite{Ay, Guan, E, K, N, RJ, TI, U, VK, XJ}, mainly by computational simulations combined with semi-analytic methods. 

Very recently in \cite{Y1, Y2, Y3} the authors rigorously proved the dissipative dynamics and the exponential synchronization of the memristive Hindmarsh-Rose neural networks and FitzHugh-Nagumo neural networks with the partly diffusive PDE-ODE models and linear interneuron couplings. Note that the linear interneuron coupling can be viewed as a strong coupling for neural networks, which is mathematically amenable but may not always well reflect the biological synaptic interactions.

It is an open and challenging problem to theoretically show that partly diffusive and memristive biological neural networks with the nonlinear interneuron coupling shown in \eqref{NW} with \eqref{Ga} is dissipative and synchronizable under a threshold condition. Such a nonlinear coupling can be called weak coupling. 

In this work we shall prove a sufficient condition on the network coupling strength to ensure an exponential synchronization of the neural networks modeled in \eqref{NW}-\eqref{Ga} by the approach of showing dissipative dynamics of the system through sharp and uniform integral estimates. Moreover, the general model \eqref{NW} with the Assumptions \eqref{Asp1}-\eqref{Asp2} for neural networks can cover all the typical neuron models including Hindmarsh-Rose equations \cite{HR}, FitzHugh-Nagumo equations \cite{FH}, Hodgkin-Huxley equations \cite{HH}, FitzHugh-Rinzel equations \cite{WZ}, and Morris-Lecar equations \cite{WL}. It is worth mentioning that an effective methodology developed here from dissipative dynamics toward synchronization was originated in J.K. Hale's paper in 1997 \cite{Ha}. This approach can be extended and explored to study many other complex or artificial neural networks in a broad scope.

\section{\textbf{Formulation and Preliminaries}}

For a framework to formulate the solution and dynamics problem of the neural network system \eqref{NW}-\eqref{inc}, we define two Hilbert spaces of functions: 
$$
E = [L^2 (\gw, \mathbb{R}^{2 + \ell})]^m \quad \text{and}  \quad  \Pi = [H^1 (\gw) \times L^2 (\gw, \mathbb{R}^{\ell + 1})]^m
$$ 
where $H^1 (\gw)$ is a Sobolev space. One can call $E$ the energy space and $\Pi$ the regular space. The norm and inner-product of $L^2(\gw)$ or $E$ will be denoted by $\| \, \cdot \, \|$ and $\inpt{\,\cdot , \cdot\,}$, respectively. We use $| \, \cdot \, |$ to denote a vector norm or a set measure in Euclidean spaces $\mR^n$.

The initial-boundary value problem \eqref{NW}-\eqref{inc} can be formulated into an initial value problem of the evolutionary equation:
\begin{equation} \label{pb}
	\begin{split}
		\frac{\partial g}{\partial t} = &\,A g + F(g), \;\; t > 0, \\
		&g(0) = g^0 \in E.
	\end{split}
\end{equation}
The unknown function in \eqref{pb} is a column vector $g(t) = \text{col}\; (g_1 (t), g_2 (t), \cdots, g_m (t))$, where the component subvector 
$$
g_i (t) = \text{col}\, (u_i(t, \cdot),\, z_i(t, \cdot),\, \rho_i (t, \cdot)), \quad \text{for} \;  1 \leq i \leq m,
$$ 
characterizes the dynamics of the neuron $\mathcal{N}_i$, for $1 \leq i \leq m$. The initial data function in \eqref{pb} is 
$$
g(0) = g^0 = \text{col}\; (g_1^0, \,g_2^0, \cdots, g_m^0) \quad \text{where} \;\; g_i^0 = \text{col}\,(u_i^0, \, z_i^0, \,\rho_i^0), \;\, 1 \leq i \leq m. 
$$
The energy norm $\|g(t)\|$ of a solution $g(t)$ for the evolutionary equation \eqref{pb} in the space $E$ is given by 
$$
\|g(t)\|^2 = \sum_{i=1}^m \|g_i (t)\|^2 = \sum_{i=1}^m \left(\|u_i(t)\|^2 + \|z_i(t)\|^2 + \|\rho_i(t)\|^2 \right).
$$
The closed linear operator $A$ in \eqref{pb} is defined by $A= \text{diag} \, (A_1, A_2, \cdots, A_m)$, where
\begin{equation} \label{opA}
	A_i =
	\begin{pmatrix}
		\eta \gd \quad & 0  \quad & 0  \\[3pt]
		0 \quad & \gL  \quad  & 0 \\[3pt]
		0 \quad & 0 \quad & - b I  \\[3pt]
	\end{pmatrix}_{(2 + \ell)\, \times\, (2 + \ell)}
	: \mathcal{D}(A) \rightarrow E, \quad i = 1, 2, \cdots, m,
\end{equation}
with the domain $\mathcal{D}(A) = \{g \in [H^2(\gw) \times L^2 (\gw, \mathbb{R}^{\ell + 1})]^m: \pdr u_i /\pdr \nu = 0, 1 \leq i \leq m\}$. The operator $A$ is the generator of $C_0$-semigroup $\{e^{At}\}_{t \geq 0}$ on the space $E$ and $I$ is the identity operator. By the fact that the Sobolev imbedding $H^{1}(\gw) \hookrightarrow L^6(\gw)$ is a continuous mapping for space dimension $n \leq 3$ and according to Assumptions \eqref{Asp1} and \eqref{Asp2}, the nonlinear mapping 
\begin{equation} \label{opf}
	F(g) =
	\begin{pmatrix}
		f (u_1, z_1) - k \tanh (\rho_1) u_1 - Pu_1 \sum_{j=2}^m \G (u_j)  \\[4pt]
		h(u_1, z_1)  \\[4pt]
		a u_1  \\[4pt]
		\vdots \\[4pt]
		f (u_m, z_m) - k \tanh (\rho_m)u_m - Pu_m \sum_{j=1}^{m-1} \G (u_j)  \\[4pt]
		h(u_m, z_m). \\[4pt]
		a u_m  
	\end{pmatrix}
	: \Pi \longrightarrow E
\end{equation}
is a locally Lipschitz continuous mapping. 

In this work we shall consider the weak solutions, cf. \cite[Section XV.3]{CV} and \cite[Section 4.2.3]{SY}, of this initial value problem \eqref{pb}.
\begin{definition} \label{D:wksn}
	A $(2 + \ell) m$-dimensional vector function $g(t, x)$, where $(t, x) \in [0, \tau] \times \gw$, is called a weak solution to the initial value problem of the evolutionary equation \eqref{pb}, if the following two conditions are satisfied: 
	
	\textup{(i)} $\frac{d}{dt} (g(t), \zeta) = (A^{1/2}g(t), A^{1/2}\zeta) + (F(g(t)), \zeta)$ is satisfied for almost every $t \in [0, \tau]$ and any $\zeta \in E^* = E$. 	
	
	\textup{(ii)} $g(t, \cdot) \in  C([0, \tau]; E) \cap C^1 ((0, \tau); E)$ and $g(0) = g^0$.		
	
	\noindent Here $\mathcal{D}(A^{1/2}) = \{g \in \Pi: \pdr u_i /\pdr \nu = 0,\, 1 \leq i \leq m\}$ and $E^*$ is the dual space of $E$. The bilinear $E$ vs $E^*$ dual product is in the scalar distribution sense. 
\end{definition}

The following proposition can be proved by Galerkin spectral approximation method \cite{CV} for the first statement on weak solutions and by the compactness property of the parabolic semigroup $e^{At}$ in mild solution bootstrap argument \cite[Theorem 42.12 and Corollary 42.13]{SY} for the second statement on strong solutions when $t > 0$.

\begin{proposition} \label{pps}
	For any given initial state $g^0 \in E$, there exists a unique weak solution $g(t; g^0), \, t \in [0, \tau]$, for some $\tau > 0$ may depending on $g^0$, of the initial value problem \eqref{pb} formulated from the memristive neural network equations \eqref{NW}. The weak solution $g(t; g^0)$ continuously depends on the initial data $g^0$ and satisfies 
	\begin{equation} \label{soln}
		g \in C([0, \tau]; E) \cap C^1 ((0, \tau); E) \cap L^2 ((0, \tau); \Pi).
	\end{equation}
	Moreover, for any initial state $g^0 \in E$, the weak solution $g(t; g^0)$ becomes a strong solution for $t \in (0, \tau)$, which has the regularity
	\begin{equation} \bl{ss}
		g \in C((0, \tau]; \Pi) \cap C^1 ((0, \tau); \Pi).
	\end{equation}
\end{proposition}

An infinite dimensional dynamical system \cite{CV, SY} for time $t \geq 0$ only is usually called a semiflow. Absorbing set defined below is the key concept to characterize dissipative dynamics of a semiflow on a Banach space.

\begin{definition} \label{Dabsb}
	Let $\{S(t)\}_{t \geq 0}$ be a semiflow on a Banach space $\ms{X}$. A bounded set $B^*$ of $\ms{X}$ is called an \emph{absorbing set} of this semiflow, if for any given bounded set $B \subset \ms{X}$ there exists a finite time $T_B \geq 0$ depending on $B$, such that $S(t)B \subset B^*$ for all $t  > T_B$. The semiflow is said to be \emph{dissipative} if there exists an absorbing set.
\end{definition}

The Young's inequality in a generic form will be used throughout in this paper. For any two positive numbers $x$ and $y$, if $\frac{1}{p} + \frac{1}{q} = 1$ and $p > 1, q > 1$, one has
\beq \bl{yg}
x\,y \leq \frac{1}{p} \ve x^p + \frac{1}{q} C(\ve, p)\, y^q \leq \ve x^p + C(\ve, p)\, y^q, \quad C(\ve, p) = \ve^{-q/p},
\eeq
where constant $\ve > 0$ can be arbitrarily small. The Gagliardo-Nirenberg interpolation inequalities \cite[Theorem B.3]{SY} will be exploited in a crucial step to prove the main result on neural network synchronization.

\section{\textbf{Dissipative Dynamics of the Memristive Semiflow}}

In this section, first we shall prove the global existence of weak solutions in time for the initial value problem \eqref{pb} and establish a solution semiflow of the memristive neural networks modeled by \eqref{NW}.  Then we show the existence of an absorbing set of this semiflow in the state spaces $E$, which exhibits the dissipative dynamics of this memristive neural network semiflow.

\begin{theorem} \label{T1}
	Under the Assumption \eqref{Asp1}, for any initial state $g^0 \in E$, there exists a unique global weak solution in time, $g(t; g^0) = \textup{col}\, (u_i(t), v_i(t), w_i(t), \rho_i(t): 1 \leq i \leq m), t \in [0, \infty)$, to the initial value problem \eqref{pb} of the memristive neural network equations \eqref{NW}. 
\end{theorem}

\begin{proof}
	Take the $L^2$ inner-products of the $u_i$-equation in \eqref{NW} with $C_1 u_i(t, x)$ for $1 \leq i \leq m$, with a scaling constant $C_1 > 0$ to be chosen later. Then sum them up. By using the Gauss divergence theorem and the boundary condition \eqref{nbc} to treat the Laplacian term and by the Assumprion \eqref{Asp1}, we can get		
\begin{equation} \bl{ui}
		\begin{split}
			&\frac{C_1}{2} \frac{d}{dt} \sum_{i=1}^m \|u_i (t)\|^2 + C_1 \sum_{i=1}^m \eta \|\nb u_i (t) \|^2 \\
			= &\, C_1 \sum_{i=1}^m \int_\gw \left[ f(u_i, z_i) u_i - k \tanh (\rho_i) u_i^2 - \sum_{j=1, j \neq i}^m \frac{Pu_i^2}{1 + \exp [-r (u_j - V)]}\right] dx  \\
			\leq &\, C_1 \sum_{i=1}^m \int_\gw \left[ - \ap |u_i |^4 + \gl |u_i| |z_i| + J - k \tanh(\rho_i) u_i^2 - \sum_{j=1, j \neq i}^m \frac{Pu_i^2}{1 + \exp [-r (u_j - V)]} \right] dx \\
			\leq &\, C_1 \sum_{i=1}^m \int_\gw \left[ - \ap |u_i|^4 + k |u_i |^2 + \gl |u_i| |z_i| + J \right] dx     \\ 
			\leq & - C_1 \ap \sum_{i=1}^m \int_\gw u_i^4 (t, x)\, dx + \left(C_1 k + \frac{C_1^2 \gl^2}{\ga}\right) \sum_{i=1}^m \|u_i \|^2 + \frac{\ga}{4} \sum_{i=1}^m \|z_i \|^2 + C_1 mJ\, |\gw|,
		\end{split}
\end{equation}
because $$ - \sum_{j=1, j \neq i}^m \frac{Pu_i^2}{1 + \exp [-r (u_j - V)]} \leq 0,$$ where the Young's inequality \eqref{yg} and the property $|\tanh (\rho_i)| \leq 1$ are used. Then sum up the $L^2$ inner-products of the $z_i$-equation with $z_i(t, x)$ and the $\rho_i$-equation with $\rho_i(t, x)$ in \eqref{NW}, $1 \leq i \leq m$, again by \eqref{yg}, we have
\beq \bl{zi}
	\begin{split}
		&\frac{1}{2} \frac{d}{dt} \sum_{i=1}^m \left(\| z_i (t)\|^2 + \|\rho_i(t)\|^2\right) = \sum_{i=1}^m \int_\gw \left(\langle \gL z_i, z_i \rangle + h(u_i, z_i)z_i + a u_i \rho_i - b \rho_i^2 \right) dx  \\
		\leq &\, \sum_{i=1}^m \int_\gw \left[- \ga |z_i|^2 + q |u_i|^2 |z_i| + L |z_i |  + a u_i \rho_i - b\, \rho_i^2 \right] dx \\
		\leq &\, \sum_{i=1}^m \frac{q^2}{4\ga} \int_\gw u_i^4 (t, x)\, dx - \ga \sum_{i=1}^m \|z_i\|^2 + \frac{a^2}{2b}  \sum_{i=1}^m \|u_i \|^2 - \frac{b}{2}  \sum_{i=1}^m \|\rho_i \|^2 + \frac{\ga}{4} \|z_i \|^2 + \frac{mL^2}{\ga} |\gw|  \\ 
		= &\, \sum_{i=1}^m \frac{q^2}{4\ga} \int_\gw u_i^4 (t, x)\, dx - \frac{3\ga}{4} \sum_{i=1}^m \|z_i\|^2 + \frac{a^2}{2b}  \sum_{i=1}^m \|u_i \|^2 - \frac{b}{2}  \sum_{i=1}^m \|\rho_i \|^2  + \frac{mL^2}{\ga} |\gw |.
	\end{split}
\eeq 
Both inequalities \eqref{ui} and \eqref{zi} are valid in the time interval $I_{max} (g^0) = [0, T_{max})$ of solution existence for each weak solution $g(t; g^0))$. 
	
Now we add the above two inequalities \eqref{ui} and \eqref{zi} to obtain
\beq \bl{uw}
	\begin{split}
		& \frac{1}{2}\frac{d}{dt} \sum_{i = 1}^m \left(C_1 \|u_i (t)\|^2 + \|z_i (t)\|^2 + \| \rho_i (t)\|^2\right) + C_1 \eta \sum_{i=1}^m \|\nb u_i (t)\|^2    \\[5pt]
		\leq & - C_1 \ap \sum_{i=1}^m \int_\gw u_i^4 (t, x)\, dx + \left(C_1 k + \frac{C_1^2 \gl^2}{\ga}\right) \sum_{i=1}^m \|u_i \|^2 + \frac{\ga}{4} \sum_{i=1}^m \|z_i \|^2 + C_1 mJ |\gw| \\[5pt]
		& + \sum_{i=1}^m \frac{q^2}{4\ga} \int_\gw u_i^4 (t, x)\, dx + \frac{a^2}{2b} \sum_{i=1}^m \|u_i \|^2 - \frac{3\ga}{4} \sum_{i=1}^m \|z_i\|^2 - \frac{b}{2} \|\rho_i \|^2 + \frac{mL^2}{\ga} |\gw | \\[5pt]
		= &- \left(C_1 \ap -  \frac{q^2}{4\ga}\right) \sum_{i=1}^m \int_\gw u_i^4 (t, x)\, dx + \left(C_1 k +  \frac{C_1^2 \gl^2}{\ga} + \frac{a^2}{2b}\right) \sum_{i=1}^m \int_\gw u_i^2 (t, x)\, dx \\[5pt]
		& - \frac{\ga}{2} \|z_i \|^2 - \sum_{i=1}^m \frac{b}{2} \|\rho_i \|^ 2 + m\left(C_1 J + \frac{L^2}{\ga}\right) |\gw|,  \quad t \in I_{max} = [0, T_{max}).
	\end{split} 
\eeq
Choose the scaling constant $C_1$ to be
\beq \bl{C1}
	C_1 = \frac{1}{\ap} \left(1 + \frac{q^2}{4\ga} \right) \quad \text{so that} \quad C_1 \ap - \frac{q^2}{4\ga} = 1.
\eeq
With this choice, from \eqref{uw} it follows that 
	\beq \bl{u1w}
	\begin{split}
		& \frac{1}{2}\frac{d}{dt} \sum_{i = 1}^m \left(C_1 \|u_i\|^2 +  \|z_i \|^2  + \|\rho_i \|^2\right) + C_1 \eta \sum_{i=1}^m \|\nb u_i\|^2 \\[5pt]
		& + \sum_{i=1}^m \int_\gw u_i^4 (t, x)\, dx - \left(C_1 k + \frac{C_1^2 \gl^2}{\ga} + \frac{a^2}{2b} \right) \sum_{i=1}^m \int_\gw u_i^2 (t, x)\, dx \\[5pt]
		& + \frac{\ga}{2} \sum_{i=1}^m \|z_i \|^2 + \frac{b}{2} \sum_{i=1}^m \|\rho_i \|^2 \leq m \left(C_1 J + \frac{L^2}{\ga}\right) |\gw |,  \quad  t \in I_{max} = [0, T_{max}).
	\end{split}
	\eeq
By completing square, we have
\beq \bl{kiq}
	\begin{split}
		& \sum_{i=1}^m \int_\gw u_i^4 (t, x)\, dx - \left(C_1 k + \frac{C_1^2 \gl^2}{\ga} + \frac{a^2}{2b} \right) \sum_{i=1}^m \int_\gw u_i^2 (t, x)\, dx \\
		= &\, \sum_{i=1}^m \int_\gw \left(u_i^4 (t, x) - \left(C_1 k + \frac{C_1^2 \gl^2}{\ga} + \frac{a^2}{2b} \right) u_i^2 (t, x)\right) dx \\
		= &\, \sum_{i=1}^m \int_\gw \left(u_i^2(t, x) -  \frac{1}{2} \left(C_1 k + \frac{C_1^2 \gl^2}{\ga} + \frac{a^2}{2b} + 1\right) \right)^2 dx  \\
		& + \sum_{i=1}^m \|u_i \|^2 - \frac{m}{4} \left(C_1 k + \frac{C_1^2 \gl^2}{\ga} + \frac{a^2}{2b} + 1\right)^2 |\gw |  \\
		\geq &\, \sum_{i=1}^m \|u_i \|^2 - \frac{m}{4} \left(C_1 k + \frac{C_1^2 \gl^2}{\ga} + \frac{a^2}{2b} + 1\right)^2 |\gw |.
	\end{split}
\eeq
Substitute \eqref{kiq} in \eqref{u1w}. It yields the inequality
\beq \bl{u2w}
	\begin{split}
		& \frac{1}{2}\frac{d}{dt} \sum_{i = 1}^m \left(C_1 \|u_i\|^2 +  \| z_i \|^2  + \|\rho_i \|^2\right) + C_1 \eta \sum_{i=1}^m \|\nb u_i\|^2  \\
		& + \sum_{i=1}^m  \left(\|u_i \|^2 + \frac{\ga}{2} \|z_i \|^2 + \frac{b}{2} \|\rho_i \|^2\right) \leq  C_2 m |\gw |, \quad  t \in I_{max}.
	\end{split}
\eeq
Denote by
\beq \bl{C2}
	C_2 = C_1J + \frac{L^2}{\ga} + \frac{1}{4} \left(C_1 k + \frac{C_1^2 \gl^2}{\ga} + \frac{a^2}{2b} + 1\right)^2. 
\eeq
We can remove the nonnegative term $C_1 \eta \sum_{i=1}^m \|\nb u_i\|^2$ from \eqref{u2w} to obtain the Gronwall-type differential inequality:  
\beq \bl{GZ}
	\begin{split}
		&\frac{d}{dt} \sum_{i = 1}^m \left[ C_1\|u_i\|^2 + \| z_i \|^2 + \|\rho_i \|^2 \right] +\mu \sum_{i = 1}^m \left[C_1\|u_i\|^2 + \| z_i \|^2 + \|\rho_i \|^2 \right]  \\
		\leq &\, \frac{d}{dt} \sum_{i = 1}^m \left[ C_1\|u_i\|^2 + \|z_i \|^2 + \|\rho_i \|^2 \right] + \sum_{i=1}^m \left(2\|u_i \|^2 + \ga \|z_i \|^2 + b \|\rho_i \|^2\right) \\[5pt]
		\leq &\, 2C_2 m |\gw |,  \quad \text{for} \:\, t \in I_{max} = [0, T_{max}),
	\end{split}
\eeq
where 
\beq \bl{mu}
	\mu = \min  \left\{ \frac{2}{C_1}, \; \ga, \; b \right\} = \min \left\{\frac{8 \ap \ga}{4\ga + q^2},\; \ga,\; b \right\}.
\eeq
Now solve the differential inequality \eqref{GZ} to obtain the following bounding estimate of all the weak solutions on the maximal existence time interval $I_{max}$,
\beq \label{dse}
	\begin{split}
		& \|g(t, g^0)\|^2 = \sum_{i=1}^m \|g_i (t, g_i^0)\|^2 = \sum_{i=1}^m \left(\|u_i (t)\|^2 + \|w_i (t)\|^2 + \|\rho_i (t)\|^2 \right) \\
		\leq &\, \frac{\max \{C_1, 1\}}{\min \{C_1, 1\}}e^{- \mu \,t} \|g^0 \|^2 +  \frac{2C_2 m}{\mu \min \{C_1, 1\}} |\gw |, \quad \; t \in [0, \infty).
	\end{split}
\eeq
Here it is shown that $I_{max} = [0, \infty)$ for every weak solution $g(t, g^0)$ because the solution will never blow up at any finite time. The uniqueness of any weak solution to the initial value problem \eqref{pb} is shown in Proposition \ref{pps}. Therefore, for any initial data $g^0 = (g_1^0, \cdots, g_m^0) \in E$, there exists a unique global weak solution of the initial value problem \eqref{pb} for this memristive reaction-diffusion neural network model \eqref{NW}-\eqref{inc} in the space $E$ for time $t \in [0, \infty)$.
\end{proof}

Based on the global existence of weak solutions established in Theorem \ref{T1}, we can define the solution semiflow $\{S(t): E \to E\}_{t \geq 0}$ of the memristive neural network system \eqref{NW} to be
$$
S(t): g^0 \longmapsto g(t; g^0) = \text{col}\, (u_i (t, \cdot), z_i (t, \cdot), \rho_i (t, \cdot): 1 \leq i \leq m), \quad t \geq 0.
$$
We shall call this semiflow $\{S(t)\}_{t \geq 0}$ the \emph{memristive reaction-diffusion neural network semiflow} generated by the model equations \eqref{NW}. 

The next theorem shows that the memristive reaction-diffusion neural network semiflow $\{S(t)\}_{t \geq 0}$ is a dissipative dynamical system in the state space $E$.
\begin{theorem} \label{Eab}
There exists a bounded absorbing set for the memristive reaction-diffusion neural network semiflow $\{S(t)\}_{t \geq 0}$ in the state space $E$, which is the bounded ball 
\beq \label{Br}
	B^* = \{ g \in E: \| g \|^2 \leq K\}
\eeq 
where the constant 
\beq \bl{K}
	K = 1 +   \frac{2C_2 m}{\mu \min \{C_1, 1\}} |\gw |,
\eeq
and the positive constants $C_1$ and $C_2$ are given in \eqref{C1} and \eqref{C2}.
\end{theorem}

\begin{proof}
This is the consequence of the global uniform estimate \eqref{dse} shown in Theorem \ref{T1}, which implies that
\beq \label{lsp}
	\limsup_{t \to \infty} \|g(t, g^0)\|^2 = \limsup_{t \to \infty} \, \sum_{i=1}^m \|g_i(t, g_i^0)\|^2 < K 
\eeq
for all weak solutions of \eqref{pb} with any initial data $g^0$ in $E$. Moreover, for any given bounded set $B = \{g \in E: \|g \|^2 \leq \mathcal{R}\}$ in $E$, there exists a finite time 
$$
	T_B = \frac{1}{\mu} \log^+ \left(\mathcal{R}\,\frac{\max \{C_1, 1\}}{\min \{C_1, 1\}}\right)
$$
such that all the solution trajectories started at the initial time $t = 0$ from the set $B$ will permanently enter the bounded ball $B^*$ shown in \eqref{Br} for $t > T_B$.  Therefore, the bounded ball $B^*$ is an absorbing set in $E$ for the semiflow $\{S(t)\}_{t \geq 0}$ so that this memristive neural network semiflow is dissipative.
\end{proof}

\section{\textbf{Higher-Order and Pointwise Ultimate Bounds}}

We shall further prove an ultimate uniform bound of the membrane potential functions $\{ u_i (t): 1 \leq i \leq m\}$ for all the weak solutions in the higher-order integrable space $L^4(\gw)$. Note that Proposition \ref{pps} and Theorem \ref{T1} together show that any weak solution $g(t) \in C((0, \infty), \Pi)$ so that the component function $u_i(t) \in C((0, \infty), H^1(\gw)) \subset C((0, \infty), L^6(\gw)) \subset C((0, \infty), L^4(\gw))$.

\begin{theorem} \bl{T4} 
	There exists a constant $Q > 0$ such that for any initial data $g^0 \in E$, the membrane potential components $u_i, 1 \leq i \leq m$, of the weak solution $g(t, g^0) = (g_1 (t), \cdots, g_m (t))$ of the initial value problem \eqref{pb} for the memristive reaction-diffusion neural network $\mathcal{NW}$ are ultimately uniform bounded in the space $L^4(\gw)$ and
\beq \bl{Lbd}
	\limsup_{t \to \infty}\, \sum_{i=1}^m \|u_i(t)\|^4_{L^4} < Q. 
\eeq
\end{theorem}

\begin{proof}
Take the $L^2$ inner-product of the $u_i$-equation in \eqref{NW} with $u_i^3 (t, \cdot), 1 \leq i \leq m$, and sum them up. By the boundary condition \eqref{nbc} and Assumption \eqref{Asp1}, we have
\beq \label{uL4} 
	\begin{split}
		&\frac{1}{4}\, \frac{d}{dt} \sum_{i=1}^m \|u_i(t)\|^{4}_{L^{4}} + 3\eta \sum_{i=1}^m \|u_i \nb u_i \|^2_{L^2} \\
		= &\, \sum_{i=1}^m \int_\gw (f(u_i, x) u_i^3 - k \tanh (\rho_i) u_i^4)\, dx - \sum_{i=1}^m \sum_{j=1}^m \int_\gw \frac{Pu_i^4}{1 + \exp [-r (u_j - V)]}\, dx \\
		\leq &\, \sum_{i=1}^m \int_\gw \left(- \ap u_i^6 + \gl u_i^3 |z_i|  + J u_i^3  + k u_i^4 \right) dx, \quad t > 0.
	\end{split}
\eeq
By Young's inequality \eqref{yg}, it is seen that 
\beq \bl{vy} 
	\gl u_i^3 |z_i|  + J u_i^3  + k u_i^4 \leq \left(\frac{\ap}{4} u_i^6 + \frac{\gl^2}{\ap}z_i^2 \right) + \left(\frac{\ap}{4} u_i^6 + \frac{J^2}{\ap}\right)+ \left(\frac{\ap}{4}u_i^6 + \frac{64}{27\ap^2} k^3 \right),
\eeq
where the last two terms in a bracket come from 
$$
	k u_i^4 = \left[\frac{3\ap}{8} u_i^6\right]^{2/3} \left[\frac{64}{9\ap^2}k^3 \right]^{1/3} \leq \frac{2}{3} \left(\frac{3\ap}{8} u_i^6\right) + \frac{1}{3} \left(\frac{64}{9\ap^2}k^3 \right) = \frac{\ap}{4}u_i^6 + \frac{64}{27\ap^2} k^3.
$$
Note that 
\beq \bl{L46}
	u_i^4 \leq \frac{1}{3} + \frac{2}{3} u_i^6 \leq 1 + u_i^6 \quad \text{so that} \; - u_i^6 \leq - u_i^4 + 1.
\eeq
From \eqref{uL4} wherein we can use the inequalities \eqref{vy} and \eqref{lsp}, it follows that 
\beq \label{u3}
	\begin{split}
		&\frac{1}{4} \, \frac{d}{dt} \sum_{i=1}^m \|u_i(t)\|^{4}_{L^{4}} + 3\eta \sum_{i=1}^m \|u_i \nb u_i \|^2 \\
		\leq &\, \sum_{i=1}^m \int_\gw \left[ - \ap u_i^6 + \left(\frac{\ap}{4} u_i^6 + \frac{\gl^2}{\ap}z_i^2 \right) + \left(\frac{\ap}{4} u_i^6 + \frac{J^2}{\ap}\right)+ \left(\frac{\ap}{4}u_i^6 + \frac{64}{27\ap^2} k^3 \right) \right] dx \\
		= &\, \sum_{i=1}^m  \left(- \frac{\ap}{4} \int_\gw u_i^6 \, dx + \frac{\gl^2}{\ap} \|z_i (t)\|^2 \right) + m \left(\frac{J^2}{\ap} + \frac{64 k^3}{27\ap^2} \right) |\gw|  \\
		\leq &\, \sum_{i=1}^m  \left(- \frac{\ap}{4} \int_\gw u_i^4 \, dx + \frac{\gl^2}{\ap} \|z_i (t)\|^2 \right) + m \left(\frac{\ap}{4} + \frac{J^2}{\ap} + \frac{64 k^3}{27\ap^2} \right) |\gw|  \\
		< &\, - \frac{\ap}{4} \sum_{i=1}^m \|u_i (t) \|^4_{L^4} + \frac{\gl^2}{\ap}K + m \left(\frac{\ap}{4} + \frac{J^2}{\ap} + \frac{64 k^3}{27\ap^2} \right) |\gw|,   
	\end{split}
\eeq
for time $t$ sufficiently large, where the constant $K$ is given in \eqref{K}, which is valid for all the weak solutions. Consequently, with the nonnegative gradient terms removed, the differential inequality \eqref{u3} shows that
\beq \bl{Gu4} 
	\begin{split}
		\frac{d}{dt} \sum_{i=1}^m \|u_i(t)\|^{4}_{L^{4}} + \ap \sum_{i=1}^m \|u_i (t) \|^4_{L^4} &\leq \frac{4\gl^2}{\ap}K + m \left(\ap + \frac{4J^2}{\ap} + \frac{256 k^3}{27\ap^2} \right) |\gw|. \\
		&\, < \frac{4\gl^2}{\ap}K + m \left(\ap + \frac{4J^2}{\ap} + \frac{10 k^3}{\ap^2} \right) |\gw|,
	\end{split}
\eeq
for $t > T(g^0)$, where $T(g^0) > 0$ is a finite time when the solution trajectory started from the initial state $g^0$ be absorbed into the absorbing set $B^*$ in the state space $E$, as shown in Theorem \ref{Eab}. 
	
By the parabolic regularity stated in Proposition \ref{pps}, for any weak solution $g(t, g^0)$ one has $u_i (T(g^0)) \in H^1 (\gw) \subset L^4 (\gw)$ for $1 \leq i \leq m$. Then the second statement in Proposition \ref{pps} shows that any weak solution has the regularity 
$$
	\sum_{i=1}^m u_i (t) \in C([T(g^0), \infty), H^1 (\gw)) \subset C([T(g^0), \infty), L^4 (\gw)). 
$$
Apply the Gronwall inequality to \eqref{Gu4}. It results in the bounding estimate of all the $u_i$ components, $1 \leq i \leq m$, in the space $L^4 (\gw)$ as follows:
\beq \bl{L4B}
	\begin{split}
		\sum_{i=1}^m \|u_i(t)\|^{4}_{L^4} &\,< e^{- \ap (t - T(g^0))} \sum_{i=1}^m \|u_i (T(g^0))\|^{4}_{L^4} \\
		&\, + \frac{4\gl^2}{\ap^2}K + m \left(1 + \frac{4J^2}{\ap^2} + \frac{10 k^3}{\ap^3} \right) |\gw|,  \quad \text{for} \;\, t \geq 0,
	\end{split}
\eeq
Therefore \eqref{Lbd} is proved with 
\beq \bl{Q}
	Q = 1 + \frac{4\gl^2}{\ap^2}K + m \left(1 + \frac{4J^2}{\ap^2} + \frac{10 k^3}{\ap^3} \right) |\gw|.
\eeq 
which is a constant independent of any initial data.
\end{proof}

The pointwise estimation in the following theorem will be used to deal with the nonlinear weak coupling toward exponential synchronization featured in this work.

\begin{theorem} \bl{T5} 
	There exists a constant $G > 0$ such that for any initial data $g^0 \in E$, the membrane potential component $u_i(t,x), 1 \leq i \leq m$, of the weak solution $g(t, g^0) = (g_1 (t), \cdots, g_m (t))$ of the initial value problem \eqref{pb} for the memristive reaction-diffusion neural network $\mathcal{NW}$ is ultimately uniform bounded in the space $\mathbb{R}$ and
\beq \bl{Rbd}
	\limsup_{t \to \infty}\, \sum_{i=1}^m |u_i(t, x)|_\mathbb{R} < G, \quad \text{for} \;\; x \in \gw. 
\eeq
\end{theorem}

\begin{proof}
	Similar to \eqref{ui} in the proof of Theorem \ref{T1}, we can multiply the $u_i$-equation in \eqref{NW} by $C_1 u_i (t, x)$ for $1 \leq i \leq m$, where $C_1$ is the same constant given in \eqref{C1}, and sum them up to get
\begin{equation} \bl{su} 
		\begin{split}
		&\frac{C_1}{2} \frac{d}{dt} \sum_{i=1}^m |u_i (t, x))|^2 + C_1 \sum_{i=1}^m \eta |\nb u_i (t,x)|^2 \\
		\leq & \sum_{i=1}^m \left[- C_1 \ap\, u_i^4 (t, x) + \left(C_1 k + \frac{C_1^2 \gl^2}{\ga}\right) u_i^2 (t,x) + \frac{\ga}{4} z_i^2 (t,x) \right] + C_1 mJ.
		\end{split}
\end{equation} 
Similar to \eqref{zi}, by multiplication and summation but without integration, we have
\beq \bl{sz}
	\begin{split}
	&\frac{1}{2} \frac{d}{dt} \sum_{i=1}^m \left(| z_i (t, x)|^2 + |\rho_i(t, x)|^2\right) \\
	\leq & \sum_{i=1}^m \left[\frac{q^2}{4\ga} u_i^4 (t, x) - \frac{3\ga}{4} z_i^2 (t, x) + \frac{a^2}{2b}u_i^2 (t, x) - \frac{b}{2} \rho_i^2 (t, x) \right] + \frac{mL^2}{\ga} .
	\end{split}
\eeq
Then parallel to the steps from \eqref{uw} through \eqref{GZ} in the proof of Theorem \ref{T1}, one can reach the pointwise differential inequality
\beq \bl{PGZ}
	\begin{split}
		&\frac{d}{dt} \sum_{i = 1}^m \left[ C_1 u_i^2(t,x) + z_i^2(t,x) + \rho_i^2(t,x) \right]  \\
		+ &\mu \sum_{i = 1}^m \left[C_1 u_i^2(t,x) + z_i^2(t, x) + \rho_i^2(t, x) \right] \leq 2C_2 m, \quad t > 0, \; x \in \gw,
	\end{split}
\eeq
where $C_2$ and $\mu$ are two universal positive constants given in \eqref{C2} and \eqref{mu} respectively. It follows that
\beq \bl{ubd}
	\sum_{i=1}^m \left(u_i^2 (t, x) + z_i^2 (t, x)+ \rho_i^2 (t, x) \right) \leq \frac{\max \{C_1, 1\}}{\min \{C_1, 1\}}e^{- \mu \,t} |g^0 (t, x)|^2 +  \frac{2C_2 m}{\mu \min \{C_1, 1\}},
\eeq
for $t > 0,x \in \gw$, which implies that \eqref{Rbd} is valid with a uniform constant
\beq \bl{G}
	G = \left[1 +  \frac{2C_2 m}{\mu \min \{C_1, 1\}} \right]^{1/2},
\eeq  
which is independent of any initial data. 
\end{proof}

\section{\textbf{Synchronization of Memristive Reaction-Diffusion Neural Networks}} 

In this section, we shall prove the main result on the synchronization of the memristive reaction-diffusion neural networks described by \eqref{NW} in the state space $E$. This result provides a quantitative threshold condition for the interneuron coupling strength to reach the neural network synchronization.

\begin{definition}
For a model evolutionary equation of a general neural network called \emph{GNW}, such as \eqref{pb} formulated from the memristive reaction-diffusion equations \eqref{NW}, we define the asynchronous degree of this neural network in a state space (as a Banach space) $W$ to be
$$
	deg_s \,(\text{\emph{GNW}})= \sum_{1\, \leq i \,< j\, \leq \,m} \left\{ \sup_{g_i^0, \, g_j^0\,  \in \, W} \, \left\{\limsup_{t \to \infty} \, \|g_i (t; g^0_i) - g_j (t; g^0_j)\|_W \right\}\right\}
$$ 
where $g_i (t)$ and $g_j (t)$ are any two solutions of this model evolutionary equation with the initial states $g_i^0$ and $g_j^0$ respectively for two neurons $\mathcal{N}_i$ and $\mathcal{N}_j$, $1 \leq i, j \leq m$, in the network. The neural network is said to be asymptotically synchronized if 
$$
	deg_s \,(\text{\emph{GNW}}) = 0.
$$
If the asymptotic convergence to zero of the difference norm for any two neurons in the network admits a uniform exponential rate, then the neural network is called exponentially synchronized. 
\end{definition}

Introduce the neuron difference functions: For $i, j = 1, \cdots, m$, define 
\begin{gather*}
	U_{ij} (t,x) = u_i(t,x) - u_j (t,x), \\
	Z_{ij} (t,x) = z_i(t,x) - z_j (t,x),  \\
	R_{ij} (t,x) = \rho_i(t,x) - \rho_j (t,x).
\end{gather*}
Given any initial state $g^0 = \text{col}\, (g_1^0, \cdots, g_m^0)$ in the space $E$, the difference between any two solutions of \eqref{pb} associated with two neurons $\mathcal{N}_i$ and $\mathcal{N}_j$ in the network is what we consider: 
$$
	g_i (t, g_i^0) - g_j (t, g_j^0) = \text{col}\, (U_{ij}(t, \cdot ), Z_{ij}(t, \cdot ), R_{ij}(t, \cdot)), \quad t \geq 0.
$$
By subtraction of the three governing equations for the $j$-th neuron from the corresponding governing equations for the $i$-th neuron in \eqref{NW}, we obtain the following differencing reaction-diffusion equations. For $i, j = 1, \cdots, m$,
\beq \bl{dHR} 
\begin{split}
	\frac{\pdr U}{\pdr t} = &\, \eta \gd U + f(u_i, z_i) - f(u_j, z_j) - k (\tanh (\rho_i)u_i - \tanh (\rho_j) u_j)  \\
	& - P \left[u_i \sum_{\nu =1}^m \G (u_\nu) - u_j \sum_{\nu =1}^m \G (u_\nu)\right]. \\
	\frac{\pdr Z}{\pdr t} = &\, \gL Z + h(u_i, z_i) - h(u_j, z_j), \\
	\frac{\pdr R}{\pdr t} = &\, aU - bR.
\end{split}
\eeq
Here and after, for any given $i$ and $j$, we shall simply write $U(t, x) = U_{ij}(t, x), Z (t, x) = Z_{ij}(t, x), R(t, x) = R_{ij}(t, x)$ as a notational convenience. 

The following exponential synchronization theorem is the main result of this paper.
\begin{theorem} \bl{ThM}
For memristive reaction-diffusion neural networks $\mathcal{NW}$ with the model \eqref{NW}-\eqref{nbc} and the Assumptions \eqref{Asp1}-\eqref{Asp2}, if the following threshold condition is satisfied by the coupling strength coefficient $P$,  
\beq \bl{SC} 
	\begin{split}
	P >  &\frac{1 + \exp [r (G + |V|)]}{m} \,\times   \\
	& \times \left[\gb + \frac{\gb^2 + \xi^2}{\ga} + k + \frac{a^2}{b} + 2\sqrt{2Q}\, C^* \left(\frac{k^2}{b}+ \frac{2\xi^2}{\ga}\right) + \frac{64 Q^2 C^{*4}}{\eta^3}\left[\frac{k^2}{b} + \frac{2\xi^2}{\ga}\right]^4\right] 
	\end{split}
\eeq
where the positive constant $Q$ and $G$ are given in \eqref{Q} and \eqref{G} respectively and $C^*$ is a coefficient in the Gagliardo-Nirenberg interpolation inequality \eqref{intp}, then the memristive neural network $\mathcal{NW}$ is exponentially synchronized in the state space $E$ at a uniform exponential rate 
\beq \bl{rate}
	\delta (P) = \min \left\{b,\, \ga,\, 2\left(\frac{mP}{1 + \exp [r (G + |V|)]} - \kappa \right) \right\},
\eeq
where the positive constant $\kappa$ is 
\beq  \bl{kap}
	\kappa = \gb + \frac{\gb^2 + \xi^2}{\ga} + k + \frac{a^2}{b} + 2\sqrt{2Q}\, C^* \left(\frac{k^2}{b} + \frac{2\xi^2}{\ga}\right) + \frac{64 Q^2 C^{*4}}{\eta^3}\left[\frac{k^2}{b} + \frac{2\xi^2}{\ga}\right]^4.
\eeq
\end{theorem}

\begin{proof}
The proof will go through three steps. 
	
Step 1. Take the $L^2$ inner-products of the first equation in \eqref{dHR} with $U(t)$, the second equation in \eqref{dHR} with $Z(t)$, and the third equation in \eqref{dHR} with $R(t)$. Then sum them up and use the Assumptions \eqref{Asp1}-\eqref{Asp2} to get
\beq \bl{eG} 
	\begin{split}
		&\frac{1}{2} \frac{d}{dt} (\|U (t)\|^2 + \|Z (t)\|^2 + \|R (t)\|^2) + \eta \|\nb U(t)\|^2 + \ga\,\|Z(t)\|^2 + b \|R(t)\|^2  \\
		&+ P \int_\gw \left[u_i \sum_{\nu =1}^m \G (u_\nu) - u_j \sum_{\nu =1}^m \G (u_\nu)\right] U(t, x)\, dx \\
		= &\, \int_\gw (f(u_i, z_i) - f(u_j, z_j)) U\, dx - \int_\gw k (\tanh (\rho_i)u_i - \tanh (\rho_j) u_j) U]\, dx \\
		&\, + \int_\gw (h(u_i, z_i) - h(u_j, z_j)) Z\, dx  + \int_\gw aUR\, dx \\
		\leq &\, \int_\gw  \frac{\pdr f}{\pdr s} \left(\zeta u_i + (1- \zeta) u_j \right) U^2\, dx + \int_\gw \frac{\pdr f}{\pdr \gs} \left(\varsigma z_i +  (1 - \varsigma) z_j \right)UZ\,dx \\
		&\, + \int_\gw  \frac{\pdr h}{\pdr s} \left(\epsilon u_i + (1- \epsilon) u_j \right) UZ\, dx  \\
		&\, - k \int_\gw \left[\text{sech}^2 (\varepsilon \rho_i + (1 - \varepsilon)\rho_j) R\, u_i U + \tanh (\rho_j)U^2 \right] dx+ \int_\gw aUR\, dx \\
		\leq &\, \int_\gw \left(\gb (U^2 + |U Z|) + \xi (|u_i| + |u_j| + 1) |UZ| + k( |u_i| RU + U^2) + a UR\right) dx. \\
		\leq &\, \int_\gw \left(\gb + \frac{\gb^2 + \xi^2}{\ga} + k + \frac{a^2}{b}\right) U^2(t, x) dx + \frac{\ga}{4}\|Z(t)\|^2 + \frac{b}{4} \|R(t)\|^2  \\
		&\, + \int_\gw \xi (|u_i| + |u_j|) |UZ|\, dx + \int_\gw k |u_i| RU\, dx, \quad t > 0.	 
	\end{split}
\eeq 
where the mean value theorem in differentiation and the hyperbolic function properties $|\tanh (\rho_j)| \leq 1, \text{sech}^2 (\ap \rho_i + (1 - \ap)\rho_j) \leq 1$ are used and the numbers $\zeta, \varsigma, \epsilon, \varepsilon \in [0,1]$. 
	
Step 2. We treat the last two integral terms on the right-hand side of the inequality \eqref{eG}. By the H\"{o}lder inequality, 
\beq \bl{RU}
	\begin{split} 
		& \int_\gw k |u_i| RU\, dx \leq k \int_\gw \left(\frac{b}{4k} R^2(t, x) + \frac{k}{b} u_i^2 (t, x) U^2(t, x)\right) dx \\
		\leq &\, \frac{b}{4} \|R(t)\|^2 +  \frac{k^2}{b} \left[\int_\gw u_i^4 (t, x)\,dx\right]^{1/2} \left[\int_\gw U^4 (t, x)\,dx\right]^{1/2} \\
		= &\, \frac{b}{4} \|R(t)\|^2 +  \frac{k^2}{b} \|u_i (t)\|^2_{L^4} \|U(t)\|^2_{L^4}, \quad t > 0.
	\end{split} 
\eeq
Similarly we have
\beq \bl{UZ}
	\begin{split} 
		&\int_\gw \xi (|u_i| + |u_j|) |UZ|\, dx \leq \xi \int_\gw \left(\frac{\ga}{4\xi} Z^2(t, x) + \frac{2\xi}{\ga} (u_i^2 + u_j^2) U^2(t, x)\right) dx \\
		\leq &\, \frac{\ga}{4} \|Z(t)\|^2 +  \frac{2 \xi^2}{\ga} \left(\left[\int_\gw u_i^4\, dx\right]^{1/2} + \left[\int_\gw u_j^4\, dx\right]^{1/2}\right) \left[\int_\gw U^4 (t, x)\,dx\right]^{1/2} \\
		= &\, \frac{\ga}{4} \|Z(t)\|^2 +  \frac{2 \xi^2}{\ga} \left(\|u_i (t)\|^2_{L^4} + \|u_j (t)\|^2_{L^4}\right) \|U(t)\|^2_{L^4}, \quad t > 0.
	\end{split} 
\eeq
Substitute the term estimates \eqref{RU} and \eqref{UZ} into the differential inequality \eqref{eG}. We obtain
\begin{equation} \bl{mG}
		\begin{split} 
		&\frac{1}{2} \frac{d}{dt} (\|U (t)\|^2 + \|W (t)\|^2 + \|R (t)\|^2) + \eta \|\nb U(t)\|^2 + \frac{\ga}{2} \|Z(t)\|^2 + \frac{b}{2} \|R(t)\|^2   \\
		&+ P \int_\gw \left[u_i \sum_{\nu =1}^m \G (u_\nu) - u_j \sum_{\nu =1}^m \G (u_\nu)\right] U(t,x)\, dx \\
		\leq &\,\left[\gb + \frac{\gb^2 + \xi^2}{\ga} + k + \frac{a^2}{b}\right] \|U(t)\|^2  \\
		&\,+ \left[\left(\frac{k^2}{b} + \frac{2\xi^2}{\ga}\right) \|u_i(t)\|^2_{L^4}  + \frac{2\xi^2}{\ga}\|u_j(t)\|^2_{L^4}\right] \|U(t)\|^2_{L^4}, \quad  t > 0.
		\end{split} 
\end{equation}
	
The challenge is to handle the last two terms of $L^4$-norm products on the right-hand side of the above inequality \eqref{mG}. We exploit the Gagliardo-Nirenberg interpolation inequalities \cite[Theorem B.3]{SY} and \cite{BV}. It states that Sobolev embedding 
$$
	H^1 (\gw) \subset L^4 (\gw) \subset L^2 (\gw)
$$ 
implies
\beq \bl{intp}
	\begin{split}
		&\|U(t) \|^2_{L^4} \leq C^* \|U(t)\|_{H^1}^{2\theta} \|U(t)\|^{2(1 - \theta)}  \\[4pt]
		\leq &\, C^* (\|U(t)\|+ \|\nb U(t)\|)^{2\theta} \|U(t)\|^{2(1 - \theta)}  \\[2pt]
		\leq &\,C^* 2^{2\theta} (\|U(t)\|^{2\theta} + \|\nb U(t)\|^{2\theta}) \|U(t)\|^{2(1 - \theta)}  \\[2pt] 
		= &\, 2\sqrt{2}C^* \|U(t)\|^2 + 2\sqrt{2}C^* \|\nb U(t)\|^{3/2} \|U(t)\|^{2(1 - 3/4)} 
	\end{split}
\eeq
where an inequality in \cite[Theorem 4.7]{Bz} is used and the coefficient $C^*(\gw) > 0$ only depends on the spatial domain $\gw$. Here the interpolation index $\theta = 3/4$ is determined by 
$$
	- \frac{n}{4} \leq \theta \left(1 - \frac{n}{2} \right) - (1 - \theta)\, \frac{n}{2}, \quad \text{for} \; 1 \leq n = \dim \gw \leq 3,
$$
and the equality holds for $n = 3$. The interpolation inequality \eqref{intp} shows that 
\beq \bl{intU}
	\|U(t) \|^2_{L^4} \leq 2\sqrt{2}C^* \|U(t)\|^2 + 2\sqrt{2}C^* \|\nb U(t)\|^{3/2} \|U(t)\|^{1/2}.
\eeq
According to Theorem \ref{Eab} and Theorem \ref{T4}, we know that $\limsup_{t \to \infty} \|U(t)\|^2 < K$ and $\limsup_{t \to \infty} \sum_{i=1}^m \|u_i(t)\|^4_{L^4} < Q$. Thus for any given initial state $g^0 \in E$ there exists a finite time $T(g^0) \geq 0$ such that 
$$
	\sum_{i=1}^m \|u_i (t)\|^2_{L^4} < Q^{1/2}, \quad \text{for all} \;\; t > T(g^0).
$$
Therefore, by \eqref{intU} and Young's inequality \eqref{yg}, we achieve the estimate
\beq \bl{U42}
	\begin{split} 
		&\left[\left(\frac{k^2}{b} + \frac{2\xi^2}{\ga}\right) \|u_i(t)\|^2_{L^4}  + \frac{2\xi^2}{\ga}\|u_j(t)\|^2_{L^4}\right] \|U(t)\|^2_{L^4} \\
		= &\, \left[ \frac{k^2}{b} \|u_i(t)\|^2_{L^4} + \frac{2\xi^2}{\ga} \left(\|u_i(t)\|^2_{L^4} + \|u_j(t)\|^2_{L^4}\right)\right] \|U(t)\|^2_{L^4} \\ 
		\leq &\,\left(\frac{k^2}{b} + \frac{2\xi^2}{\ga}\right) Q^{1/2} 2\sqrt{2} C^* (\|U(t)\|^2 + \|\nb U(t)\|^{3/2} \|U(t)\|^{1/2})  \\
		\leq &\, \left(\frac{k^2}{b} + \frac{2\xi^2}{\ga}\right) Q^{1/2} 2\sqrt{2} C^* \|U(t)\|^2 + \eta \|\nb U(t)\|^{(3/2) \times (4/3)}   \\
		&\,+ \frac{1}{\eta^3}\left[\left(\frac{k^2}{b} + \frac{2\xi^2}{\ga}\right) Q^{1/2} 2\sqrt{2} C^* \|U(t)\|^{1/2}\right]^4 \\
		= &\,\eta \|\nb U(t)\|^2 + \left[2 \sqrt{2Q}\, C^* \left(\frac{k^2}{b} + \frac{2\xi^2}{\ga}\right) + \frac{64 Q^2 C^{*4}}{\eta^3}\left(\frac{k^2}{b} + \frac{2\xi^2}{\ga}\right)^4\right] \|U(t)\|^2, 
	\end{split} 
\eeq
for $t > T(g^0)$. 

Substitute \eqref{U42} in \eqref{mG} and then cancel the gradient terms $\eta \|\nb U(t)\|^2$ on two sides of that inequality. It follows that	
\begin{equation} \bl{MG}
		\begin{split} 
			&\frac{1}{2} \frac{d}{dt} (\|U (t)\|^2 + \|Z (t)\|^2 + \|R (t)\|^2) + \frac{\ga}{2} \|Z(t)\|^2 + \frac{b}{2} \|R(t)\|^2   \\[5pt]
			&+ P \int_\gw \left[u_i \sum_{\nu =1}^m \G (u_\nu) - u_j \sum_{\nu =1}^m \G (u_\nu)\right] U(t, x)\, dx \\
			\leq &\,\left[\gb + \frac{\gb^2 + \xi^2}{\ga} + k + \frac{a^2}{b} + 2\sqrt{2Q}\, C^* \left(\frac{k^2}{b} + \frac{2\xi^2}{\ga}\right) + \frac{64 Q^2 C^{*4}}{\eta^3}\left[\frac{k^2}{b} + \frac{2\xi^2}{\ga}\right]^4 \right] \|U(t)\|^2.
		\end{split} 
\end{equation}
	
Step 3. Another challenge is to handle the nonlinear difference term of the weak coupling on the left-hand side of the inequality \eqref{MG}. For any given $1 \leq i \neq j \leq m$, we have
\beq \bl{CD}
	\begin{split} 
		& P \int_\gw \left[u_i \sum_{\nu =1}^m \G (u_\nu) - u_j \sum_{\nu =1}^m \G (u_\nu)\right] U(t, x)\,dx \\[4pt]
		= &\,P \int_\gw \sum^m_{\nu = 1} \frac{u_i - u_j}{1 + \exp [- r (u_\nu - V)]} U(t, x)\,dx   \\[4pt]
		= &\,P \int_\gw \sum^m_{\nu = 1} \frac{U^2 (t, x)}{1 + \exp [- r (u_\nu - V)]} \,dx.  \\
	\end{split} 
\eeq
By Theorem \ref{T5} and \eqref{Rbd}, for each solution trajectory $g(t, g^0)$ there exists a finite time $\tau (g^0) > 0$ such that $\sum_{i=1}^m |u_i(t, x)|_\mathbb{R} < G$ for $t > \tau (g^0)$. Hence it holds that
\beq \bl{kbd}
	\frac{1}{1 + \exp\, [- r (u_\nu (t, x) - V)]}  \geq \frac{1}{1 + \exp\, [r (G + |V|)]}, \quad t > \tau(g^0),
\eeq
for all $1 \leq \nu \leq m$. Now substitute \eqref{CD} and \eqref{kbd} in the differential inequality \eqref{MG} on the left-hand side. We obtain

\begin{equation} \bl{HMG}
	\begin{split}
	&\frac{1}{2} \frac{d}{dt} (\|U (t)\|^2 + \|Z (t)\|^2 + \|R (t)\|^2) + \frac{\ga}{2} \|Z(t)\|^2 + \frac{b}{2} \|R(t)\|^2   \\
	& + \frac{mP}{1 + \exp\, [r (G + |V|)]} \|U(t)\|^2   \\
	= &\, \frac{1}{2} \frac{d}{dt} (\|U (t)\|^2 + \|Z (t)\|^2 + \|R (t)\|^2) + \frac{\ga}{2} \|Z(t)\|^2 + \frac{b}{2} \|R(t)\|^2   \\
	& + \frac{P}{1 + \exp\, [r (G + |V|)]} \int_\gw \sum^m_{\nu = 1} U^2 (t, x)\,dx  \\
	\leq &\, \frac{1}{2} \frac{d}{dt} (\|U (t)\|^2 + \|Z (t)\|^2 + \|R (t)\|^2) + \frac{\ga}{2} \|Z(t)\|^2 + \frac{b}{2} \|R(t)\|^2   \\
	&\, + P\, \sum^m_{\nu = 1} \int_\gw \frac{1}{1 + \exp\, [- r (u_\nu (t, x) - V)]} U^2(t, x)\,dx \\
	= &\,\frac{1}{2} \frac{d}{dt} (\|U (t)\|^2 + \|Z (t)\|^2 + \|R (t)\|^2) + \frac{\ga}{2} \|Z(t)\|^2 + \frac{b}{2} \|R(t)\|^2   \\
	&\, + P \int_\gw \left[u_i \sum_{\nu =1}^m \G (u_\nu) - u_j \sum_{\nu =1}^m \G (u_\nu)\right] U(t, x)\,dx \\
	\leq &\,\left[\gb + \frac{\gb^2 + \xi^2}{\ga} + k + \frac{a^2}{b} + 2\sqrt{2Q}\,C^*\left[\frac{k^2}{b} + \frac{2\xi^2}{\ga}\right] + \frac{64 Q^2 C^{*4}}{\eta^3}\left[\frac{k^2}{b} + \frac{2\xi^2}{\ga}\right]^4 \right] \|U(t)\|^2  
		\end{split} 
\end{equation}
for $t > \tau (g^0)$. From \eqref{HMG} and by the threshold condition \eqref{SC} stated in this theorem, it results in the following Gronwall-type inequality:
\beq \bl{Grw}
	\begin{split} 
		& \frac{d}{dt} (\|U (t)\|^2 + \|Z (t)\|^2 + \|R (t)\|^2) + \delta (P) (\|U (t)\|^2 + \|Z (t)\|^2 + \|R (t)\|^2)   \\[6pt]
		&\leq \frac{d}{dt} (\|U (t)\|^2 + \|W (t)\|^2 + \|R (t)\|^2) + \ga \|Z(t)\|^2 + b \|R(t)\|^2 + 2\left[\frac{mP}{1 + \exp\, [r (G + |V|)]} \right.  \\
		&\, \left. - \left[\gb + \frac{\gb^2 + \xi^2}{\ga} + k + \frac{a^2}{b} + 2\sqrt{2Q}\,C^*\left[\frac{k^2}{b} + \frac{2\xi^2}{\ga}\right] + \frac{64 Q^2 C^{*4}}{\eta^3}\left[\frac{k^2}{b} + \frac{2\xi^2}{\ga}\right]^4 \right]\right] \|U(t)\|^2   \\
		&\leq 0, \quad \text{for} \;\;  t > \tau (g^0).
	\end{split} 
\eeq
Denote by
$$
	\kappa = \gb + \frac{\gb^2 + \xi^2}{\ga} + k + \frac{a^2}{b} + 2\sqrt{2Q}\, C^*\left[\frac{k^2}{b} + \frac{2\xi^2}{\ga}\right] + \frac{64 Q^2 C^{*4}}{\eta^3}\left[\frac{k^2}{b} + \frac{2\xi^2}{\ga}\right]^4.
$$

Finally we can solve this linear Gronwall inequality \eqref{Grw} to reach the exponential synchronization result: For any initial state $g^0 \in E$ and any two neurons $\mathcal{N}_i$ and $\mathcal{N}_j$ in this memristive reaction-diffusion neural network model \eqref{NW}, their difference function $g_i(t; g_i^0) - g_j(t; g_j^0)$ converges to zero in the state space $E$ exponentially at a uniform convergence rate $\delta (P)$ shown below. Namely, for any $1 \leq i \neq  j \leq m$,
\beq \bl{Esyn} 
	\begin{split}
		\| g_i (t) - g_j (t) \|_E^2 & = \|u_i (t) - u_j (t)\|^2 + \|z_i(t) - z_j(t)\|^2 + \|\rho_i(t) - \rho_j(t)\|^2  \\[7pt]
		&= \|U_{ij}(t)\|^2 + \|Z_{ij}(t)\|^2 + \|R_{ij}(t)\|^2  \\[5pt]
		&\leq e^{- \delta (P) \,t } \left\|g_i^0 - g_j^0 \right\|^2  \to 0, \;\; \text{as} \;\, t \to \infty. 
	\end{split}
\eeq
Here the constant convergence rate in \eqref{Esyn} is 
\begin{equation*}
		\delta (P) = \min \left\{b,\, \ga,\, 2\left(\frac{mP}{1 + \exp\, [r (G + |V|)]} - \kappa \right) \right\},
\end{equation*}
which is exactly \eqref{rate}-\eqref{kap} stated in the threshold condition \eqref{SC} of this theorem. Hence it is proved that
\beq \bl{degs}
	deg_s (\mathcal{NW}) = \sum_{1 \,\leq \,i  \,\neq  \,j \,\leq \,m} \left\{\sup_{g^0\, \in \, E} \, \left\{\limsup_{t \to \infty} \|g_i (t) -g_j(t) \|^2_E \right\}\right\} = 0.
\eeq
The proof of this theorem is completed. 
\end{proof}

As a meaningful extension of Theorem \ref{ThM}, we can also prove the exponential synchronization of memristive reaction-diffusion neural networks denoted by $\mathbb{NW} = \{N_i: i = 1, 2, \cdots, m\}$ with the following model equations, cf. \cite{B, Co, M, PL},
\beq \bl{eNW} 
\begin{split}
	& \frac{\pdr u_i}{\pdr t} = \eta \gd u_i + f (u_i, z_i) - k \tanh (\rho_i) u_i - \sum_{j=1}^m \frac{P (u_i - u_e)}{1 + \exp [- r (u_j - V)]},  \\
	& \frac{\pdr z_i}{\pdr t}  =  \gL z_i + h(u_i, z_i),    \\[3pt]
	& \frac{\pdr \rho_i}{\pdr t} = a u_i - b \rho_i,
\end{split} 
\eeq
where the weak coupling terms involve  a constant $u_e \in \mathbb{R}$ called the reversal potential, on a bounded spacial domain $\gw$ and satisfy the same boundary conditions as specified in Section 1. 

\begin{theorem} \bl{TM}
	Assume that the nonlinear terms $f(s, \gs)$ and $h(s, \gs)$ in the memristive neural network model \eqref{eNW} are respectively scalar and vector polynomials and satisfy the same Assumptions \eqref{Asp1} and \eqref{Asp2}, then there exists a positive constant $\Psi > 0$ which depends only on the parameters including $u_e$ but independent of any initial data, such that if the threshold condition
	\beq \bl{PG}
	P > \Psi
	\eeq
is satisfied, then the solution semiflow of the memristive reaction-diffusion neural network $\mathbb{NW}$ will be exponentially synchronized in the same state space $E$ at a uniform convergence rate. 
\end{theorem}

\begin{proof}
We just briefly sketch the proof. Make the variable changes to denote $\tu_i = u_i - u_e, \, 1 \leq i \leq m$. Then the system \eqref{eNW} becomes 
\begin{equation*}
	\begin{split}
		\frac{\pdr \tu_i}{\pdr t} = \eta \gd \tu_i + f (\tu_i + u_e, z_i) &\, - k \tanh (\rho_i) (\tu_i + u_e) - \sum_{j=1}^m \frac{P\, \tu_i}{1 + \exp\, [- r (\tu_j + u_e - V)]}\, ,  \\
		& \frac{\pdr z_i}{\pdr t}  =  \gL z_i + h(\tu_i + u_e, z_i),    \\[3pt]
		& \frac{\pdr \rho_i}{\pdr t} = a (\tu_i + u_e) - b \rho_i.
	\end{split} 
\end{equation*}
This system of equations can be written as 
\beq \bl{eNW2} 
	\begin{split}
		& \frac{\pdr \tu_i}{\pdr t} = \eta \gd \tu_i + \tf (\tu_i, z_i, \rho_i) - \sum_{j=1}^m \frac{P\, \tu_i}{1 + \exp [- r (\tu_j - \tV)]},  \\
		& \frac{\pdr z_i}{\pdr t}  =  \gL z_i + \Th (\tu_i, z_i),    \\[3pt]
		& \frac{\pdr \rho_i}{\pdr t} = a \tu_i - b \rho_i + a u_e.
	\end{split} 
\eeq
where $\tV = V - u_e$ and the two new functions $\tf$ and $\Th$ are
\beq \bl{Tfh}
	\begin{split}
		\tf (\tu_i, z_i, \rho_i) &= f(\tu_i + u_e, z_i)  - k \tanh (\rho_i) (\tu_i + u_e),  \\
		\Th (\tu_i, z_i) &= h(\tu_i + u_e, z_i). 
	\end{split}
\eeq	
By expansion of the scalar and vector polynomials $f(s + u_e, \gs)$ and $h(s + u_e, \gs)$ and $|\tanh (\rho_i)| \leq 1$, using Young's inequality \eqref{yg} appropriately, it follows from the Assumptions \eqref{Asp1} and \eqref{Asp2} that the new functions $\tf$ and $\Th$ possess the properties
\beq \bl{Asp3}
	\begin{split}
		&\tf (s, \gs, \rho) s \leq  - \tap |s|^4 + \tgl |s| |\gs| + \tJ, \quad  (s, \gs, \rho) \in \mathbb{R}^{2 + \ell}, \\
		&\max \left\{\frac{\pdr \tf}{\pdr s} (s, \gs, \rho), \, \left| \frac{\pdr \tf}{\pdr \gs} (s, \gs, \rho)\right| \right\} \leq \tgb,  \; (s, \gs, \rho) \in \mathbb{R}^{2 + \ell},  \\
		& \left| \frac{\pdr \tf}{\pdr \rho} (s, \gs, \rho)\right| \leq k |s + u_e|, \; (s, \gs, \rho) \in \mathbb{R}^{2 + \ell},
	\end{split}
\eeq
and
\beq \bl{Asp4}
	\begin{split}
		&\Th (s, \gs) \gs \leq \tq |s|^2 |\gs| + \tL |\gs |,  \quad (s, \gs ) \in \mathbb{R}^{1 + \ell},  \\
		& \left| \frac{\pdr \Th}{\pdr s}(s, \gs) \right| \leq \txi (|s| + 1), \;\; \frac{\pdr \Th}{\pdr \gs}(s, \gs) = 0,  \quad (s, \gs ) \in \mathbb{R}^{1 + \ell},
	\end{split}
\eeq 
where the positive constants $\tap, \tgl, \tJ, \tgb$ in \eqref{Asp3} and $\tq, \tL, \txi $ in \eqref{Asp4} are the new parameters for the new model equations \eqref{eNW2}, which may also depend on the constant $u_e$. 
	
We notice the structural similarity and the new parameters between the Assumptions \eqref{Asp1}-\eqref{Asp2} and the properties \eqref{Asp3}-\eqref{Asp4} and we see the maneuverable differences in the term $-k \tanh(\rho_i)u_i$ of \eqref{Tfh} in \eqref{Asp3} and a spillover constant $au_e$ of the third equation of \eqref{eNW2}. 
	
Then we can conduct \emph{a priori} estimates parallel to the steps shown in Section 3 and Section 4 in the same formulated framework. It can be shown that the weak solutions of this neural network model \eqref{eNW2} exist globally in time and the solution semiflow has an absorbing set $\Tilde{B}^*$ in the same state space $E$. Specifically, we have
$$
	\Tilde{B}^* = \{\tg \in E: \|\tg \|^2 \leq K^*\}
$$
where $ g = \text{col} \, (u_1, z_1, \rho_1, \cdots, u_m, z_m, \rho_m)$ and
\beq \bl{bK}
	K^* = 1 + \frac{2C_4 m}{\Tilde{\mu} \min \{C_3, 1\}} |\gw |,
\eeq
in which
\beq \bl{prm}
	\begin{split}
		&\Tilde{\mu} = \min \left\{\frac{2}{C_3}, \; \ga, \; b\right\},  \quad C_3 = \frac{1}{\tap} \left(1 + \frac{\tq^2}{4 \ga}\right), \\
		C_4 &= C_3 \tJ + \frac{\tL^2}{\ga} + \frac{a^2 u_e}{b} + \frac{1}{4} \left(\frac{C_3^2\, \tgl^2}{\ga} + \frac{a^2}{2b} + 1\right)^2.
	\end{split}
\eeq
Moreover, the ultimate bound property holds:
$$
	\limsup_{t \to \infty} \, \sum_{i=1}^m \|\tu_i (t)\|^4_{L^4} < Q^*
$$
where
\beq \bl{tQ}
	Q^* = 1 + \frac{4 \tgl^2}{\tap} K^* + m \left(\tap + \frac{4\tJ^2}{\tap} \right) |\gw |.
\eeq
We can also get the pointwise ultimate bound:
$$
	\limsup_{t \to \infty} \, \sum_{i=1}^m |\tu_i (t, x)|_{\mathbb{R}} < G^*
$$
where
\beq \bl{tG}
	G^* = \left[1 + \frac{2C_4 m}{\Tilde{\mu} \min \{C_3, 1\}} \right]^{1/2}.
\eeq
	
Finally we define the neuron difference functions: For $1 \leq i, j \leq m\, (i \neq j)$, 
$$
	\Tilde{U}_{ij} = \tu_i - \tu_j, \qquad Z_{ij} = z_i - z_j, \qquad R_{ij} = \rho_i - \rho_j.
$$
They satisfy the following differencing reaction-diffusion equations:
\beq \bl{dNW} 
	\begin{split}
		\frac{\pdr \Tilde{U}_{ij}}{\pdr t} = &\, \eta \gd \Tilde{U}_{ij} + \tf (\tu_i, z_i, \rho_i) - \tf (\tu_j, z_j, \rho_j) - \sum_{\nu = 1}^m \frac{P (\tu_i - \tu_j)}{1 + \exp\, [-r (\tu_\nu - \tV)]} \\
		\frac{\pdr Z_{ij}}{\pdr t} = &\, \gL  Z_{ij} + \Th (\tu_i, z_i) - \Th (\tu_j, z_j), \\
		\frac{\pdr R_{ij}}{\pdr t} = &\, a\Tilde{U}_{ij} - b R_{ij}.
	\end{split}
\eeq
Parallel to the steps in the proof of Theorem \ref{ThM}, one can show that if the threshold condition \eqref{PG} is satisfied, where the threshold constant
\beq \bl{eSC}
	\begin{split}
		&\Psi =  \frac{1 + \exp\, [r (G^* + |\tV|)]}{m}\, \times  \\
		& \left[\tgb + \frac{\tgb^2 + \txi^2}{\ga} + k + \frac{a^2}{b} + 2\sqrt{2 Q^*}C^*\left(\frac{k^2}{b}+ \frac{2\txi^2}{\ga}\right) + \frac{64(Q^*)^2 (C^*)^4}{\eta^3}\left[\frac{k^2}{b} + \frac{2\txi^2}{\ga}\right]^4\right]
	\end{split}
\eeq
and the mathematical coefficient $C^*$ remains the same as in Theorem \ref{ThM}, then the solutions of this memristice neural network $\mathbb{NW}$ is exponentially synchronized in the state space $E$ at a uniform convergence rate.
\end{proof}

\section{\textbf{Examples and Numerical Simulation}}

In this section we shall provide two typical and most useful mathematical models of biological neural networks with memristors to illustrate the applications of the achieved exponential synchronization result in Theorem \ref{ThM}.

To avoid notational overlap or confusion, the parameters in the following two subsections will be attached with subscript 1 and subscript 2, respectively. 

Numerical simulation for these two types of memristive neural networks will be performed to show the synchronization convergence behavior with a relatively higher threshold and possibly lower convergence rate due to the nonlinear weak coupling of the solution trajectories in the depicted curves of their $L^2$-norms.

\subsection{\textbf{Diffusive Hindmarsh-Rose Equations with Memristor}}

Consider a model of memristive diffusive Hindmarsh-Rose neural networks with memristor \cite{HR, CPY, Y1}:
\begin{equation} \bl{HR}
	\begin{split}
		\frac{\pdr u_i}{\pdr t} & = \eta_1 \gd u_i + a_1 u_i^2 - b_1 u_i^3 + v_i - w_i - k_1\tanh (\rho_i) u_i - P u_i \sum_{j=1}^m \G (u_j), \,  \\
		\frac{\pdr v_i}{\pdr t} & = \ap_1 - \gb_1 u_i^2 - v_i,  \\
		\frac{\pdr w_i}{\pdr t} & = q_1 u_i  - r_1 w_i,  \\
		\frac{\pdr \rho_i}{\pdr t} & = c_1 u_i - \delta_1 \rho_i, 
	\end{split}
\end{equation}
for $t > 0,\; x \in \gw \subset \mathbb{R}^{n}$ ($n \leq 3$), where $1 \leq i \leq m$ and $\gw$ is a bounded domain up to three dimension with locally Lipschitz continuous boundary. The nonlinear function $\Gamma (s)$ is the same as in \eqref{Ga}.

In this system \eqref{HR}, the variable $u_i(t, x)$ refers to the membrane electric potential of a neuron cell, the variable $v_i(t, x)$ represents the transport rate of the ions of sodium and potassium through the fast channels and can be called the spiking variable, while the variables $w_i(t, x)$  called the bursting variable represents the transport rate across the neuron membrane through slow channels of calcium and some other ions. All the involved parameters $a_1, b_1, c_1, \eta_1, \ap_1, \gb_1, q_1, r_1,  \de_1, k_1$ and the coupling strength coefficient $P$ can be any positive constants. 

We impose the homogeneous Neumann boundary conditions for the $u$-component, $\frac{\pdr u}{\pdr \nu} \, (t, x) = 0,\, x \in \partial \gw$, and the initial conditions of the components are denoted by
$$
u_i^0 (x) = u_i(0, x), \;v_i^0 (x) = v_i(0, x), \; w_i^0 (x) = w_i (0, x), \; \rho_i^0 = \rho_i (0, x), \;\; 1 \leq i \leq m.
$$

For illustrating the synchronization result Theorem \ref{ThM}, we can simply check all the Assumptions in \eqref{Asp1} and \eqref{Asp2} are satisfied by this model of memristive Hindmarsh-Rose equations \eqref{HR}. The vector functions $z_i (t, x)$ in the general model \eqref{NW} in this case is
$$
z_i (t, x) = 
\begin{pmatrix}
	v_i (t, x) \\[3pt]
	w_i (t, x)  \\[3pt]
\end{pmatrix}
$$
and correspondingly the vector $\gs = \text{col}\, (\gs_v, \gs_w)$ has two components.

Verify the Assumptions \eqref{Asp1} and \eqref{Asp2}: In this model, we have the scalar function
\begin{equation*}
	f(s, \gs) = a_1 s^2 - b_1 s^3 +  \gs_v - \gs_w,
\end{equation*}
the 2-dimensional square matrix and the vector function 
\begin{equation*}
	\gL =
	\begin{pmatrix}
		- I \quad & 0  \\[3pt]
		0 \quad & - r_1 I. \\[3pt]
	\end{pmatrix},
	\quad
	h(s, \gs) = 
	\begin{pmatrix}
		\ap_1 - \gb_1 s^2  \\[3pt]
		q_1 s  \\[3pt]
	\end{pmatrix}.
\end{equation*}
We can verify that
\beq \bl{Ck1}
\begin{split}
	f(s, \gs)s &= s(a_1 s^2 - b_1 s^3 +  \gs_v - \gs_w) = a_1 s^3 -b_1 s^4 + s(\gs_v - \gs_w)  \\
	&\leq \left(\frac{3b_1}{4} |s|^4 + \frac{a_1^4}{4b_1^3}\right) - b_1 |s|^4 + |s|(|\gs_v| + |\gs_w|) \\
	&\leq -\frac{b_1}{4} |s|^4 + \sqrt{2} |s| |\gs| + \frac{a_1^4}{4b_1^3},  \quad \text{for} \;\, (s, \gs) \in \mathbb{R}^3,	
\end{split}
\eeq
and
\beq \bl{Ck2}
\begin{split}
	&\max \left\{\frac{\partial f}{\partial s}(s, \gs), \left| \frac{\partial f}{\partial \gs}(s, \gs) \right| \right\} = \max\, \{2a_1 s - 3b_1 s^2, 1\}   \\[2pt]
	\leq &\,\max\, \left\{\frac{a_1^2}{2b_1} + (2b_! - 3b_1)s^2, 1 \right\} \leq \max\, \left\{\frac{a_1^2}{2b_1}, \,1 \right\}, \quad \text{for} \;\, (s, \gs) \in \mathbb{R}^3.
\end{split}
\eeq
Therefore Assumption \eqref{Asp1} is satisfied. Moreover,
\beq \bl{Ck3}
\begin{split}
	\langle \gL\gs, \gs \rangle &= - \gs_v^2 - r_1 \gs_w^2 \leq - \min \left\{1, \, r_1\right\} |\gs|^2, \quad \text{for} \;\, \gs \in \mathbb{R}^2,  \\[4pt]
	h(s, \gs) \gs &= (\ap_1 - \gb_1 s^2)\gs_v + q_1 s \gs_w \leq \gb_1 s^2 |\gs_v| + q_1 s^2 |\gs_w| + (\ap_1 |\gs_v| + q_1 |\gs_w|)  \\[4pt]
	&\leq (\gb_1 + q_1)s^2 |\gs|  + (\ap_1 + q_1) |\gs|, \quad \text{for} \;\, (s, \gs) \in \mathbb{R}^3,
\end{split}
\eeq
and
\beq \bl{Ck4}
\begin{split}
	&\frac{\partial h}{\partial s} (s, \gs) \leq |\text{col}\, (-2\gb_1 s,\, q_1)| \leq \max\, \{2 \gb_1, q_1\} (|s | + 1),   \\[4pt]
	&\frac{\partial h}{\partial \gs} (s, \gs) = 0, \quad \text{for} \;\, (s, \gs) \in \mathbb{R}^3.
\end{split}
\eeq
Therefore Assumption \eqref{Asp2} is also satisfied. We can record the specific parameters in \eqref{Asp1} and \eqref{Asp2} for this memristive Hindmarsh-Rose neural network model as follows:
\beq \bl{spm}
\begin{split}
	&\ap =  \frac{b_1^4}{4},  \quad \gl = \sqrt{2},  \quad J = \frac{a_1^4}{4b^3}, \quad  \gb = \max\, \left\{\frac{a_1^2}{2b_1}, \,\sqrt{2}\right\},  \\
	\ga = \max\, &\left\{1, \, r_1\right\}, \quad   q =  \gb_1 + q_1, \quad L =  \ap_1 + q_1, \quad \xi = \max\, \{2 \gb_1, q_1\}.
\end{split}
\eeq
Apply the proved synchronization Theorem \ref{ThM} to this memristive diffusive Hindmarsh-Rose neural network model \eqref{HR}. Then we reach the following result.

\begin{theorem} \bl{ThHR}
	For memristive diffusive Hindmarsh-Rose neural networks with the model \eqref{HR}, if the threshold condition \eqref{SC} with the parameters in \eqref{spm} is satisfied by the coupling strength coefficient $P$, then the neural network is exponentially synchronized in the state space $E = [L^2 (\gw, \mathbb{R}^4)]^m$ at a uniform exponential convergence rate $\delta (P)$ shown in \eqref{rate} with the parameters given in \eqref{spm}.
\end{theorem}

We numerically solve the differential equations of the memristive Hindmarsh-Rose neural network with the model \eqref{HR} in a two-dimensional square domain. And we use the finite difference method for the numerical scheme programmed in Python.

Choose the following parameters in the model \eqref{HR}:
\begin{gather*}
	m = 4,\;\ \eta_2=5,\; \ a_1=1,\;\ b_1=2, \, k_1=0.3, \;\ V=0.5, \;\ r=0.1, \\
	\alpha_1=0.4, \;\ \beta_1=0.06,\;\  q_1=0.2,\;\  r_1=4, \;\  c_1=1, \;\ \delta_1=7.
\end{gather*}
Take time-step to be 0.00025 and spatial-step to be 1 on a $32*32$ membrane. We compute and show the $L^2$-norm curves of the neuron membrane potential variable $u_i$, the spiking variable $v_i$, the bursting variable $w_i$ and the memductance variable $\rho_i$ in Figure \ref{figHR1} to Figure \ref{figHR4}. Lastly the pairwise difference of vector solutions $g_i, \ i=1,2,3,4$ in the energy space $E$ is shown in Figure \ref{figHR5}. 

In Figure \ref{figHR1} to Figure \ref{figHR4}, with a comparison between results after $666$ iterations and results after $2000$ iterations, one can observe the synchronization tendency of the four characterizing variables $(u_i, v_i, w_i, \rho_i)$ among the neurons in the simulated mimristive Hindmarsh-Rose neural network. From Figure \ref{figHR5}, we observe that the $L^2$-norms of pairwise differences $\|g_i - g_j\|$ tend to $0$. 

We also calculate the following key constants involved in Theorem \ref{ThM} based on our selection of parameters, rounding up to 2 digits.
\begin{gather*}
	C_1 = 0.25, \quad C_2 = 0.44, \quad \mu = 4.0,  \quad K = 3630.45, \quad  Q = 23719.02, \quad \\
	G = 2.12, \quad C^* = 0.4,  \quad  \kappa = 16.69, \quad  P = 19.60, \quad \delta = 4.0.
\end{gather*}
The constant $C^*$ from Gargliardo-Nirenberg inequality is chosen to be $0.4$ based on \cite{BV}.

Table \ref{tab:tableHR1} to Table \ref{tab:tableHR4} list the sampled values of the four components $u_i, v_i, w_i$, and $\rho_i$ of the simulated solution $g_i$ at one same point in the domain at the initial time $t=0$, at the 200 and 2000 time-step. It is seen that with a big difference on the initial values, after a certain time, the values of $u_i$, $v_i$,  $w_i$, and $\rho_i$ tend to be close to each other between various neurons. 
\begin{figure}%
	\centering
	\subfloat{{\includegraphics[width=12cm]{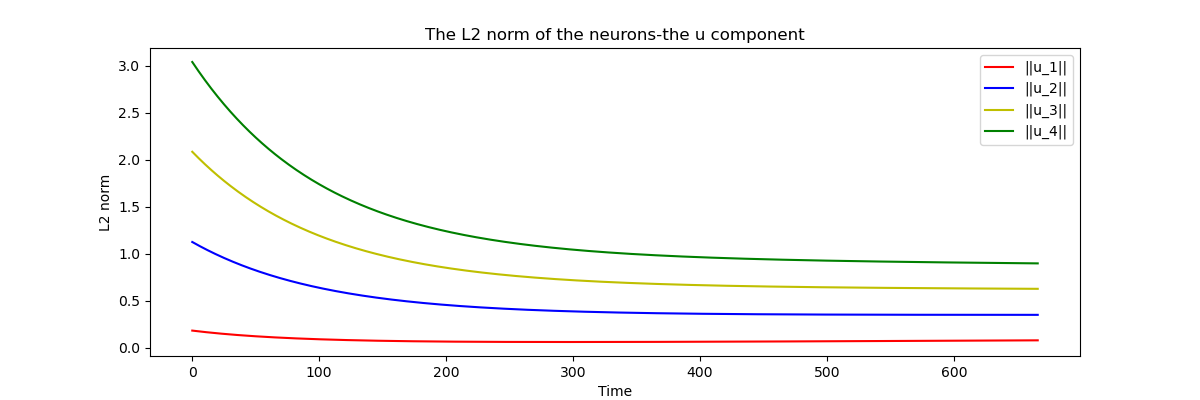}} }%
	\qquad
	\subfloat{{\includegraphics[width=12cm]{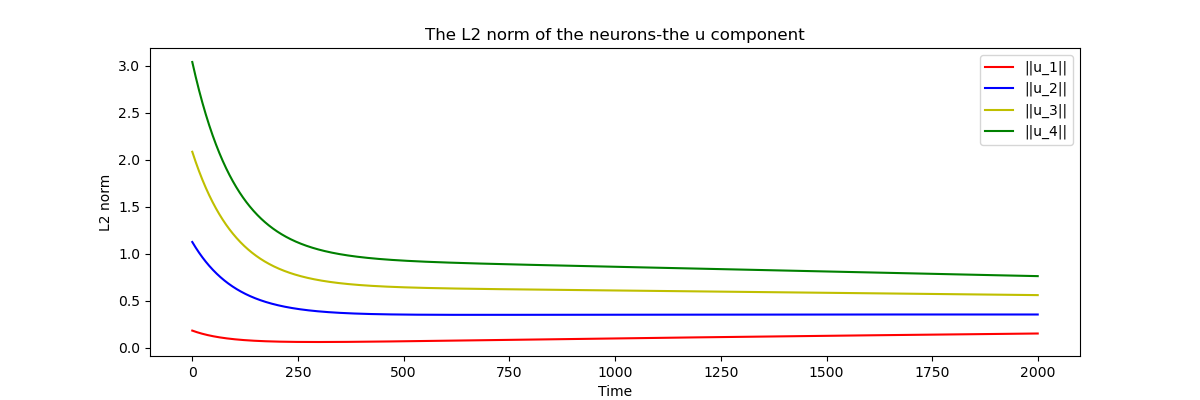} }}%
	\caption{The $L^2$ norm of the neurons component $u_i$ after 666 iterations (upper figure) and after 2000 iterations (lower figure)}%
	\label{figHR1}%
\end{figure}

\begin{figure}%
	\centering
	\subfloat{{\includegraphics[width=12cm]{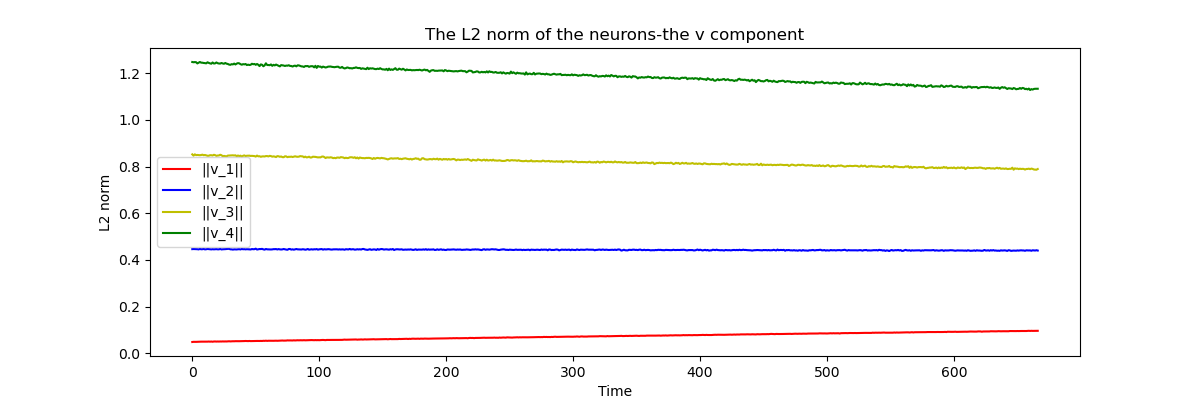} }}%
	\qquad
	\subfloat{{\includegraphics[width=12cm]{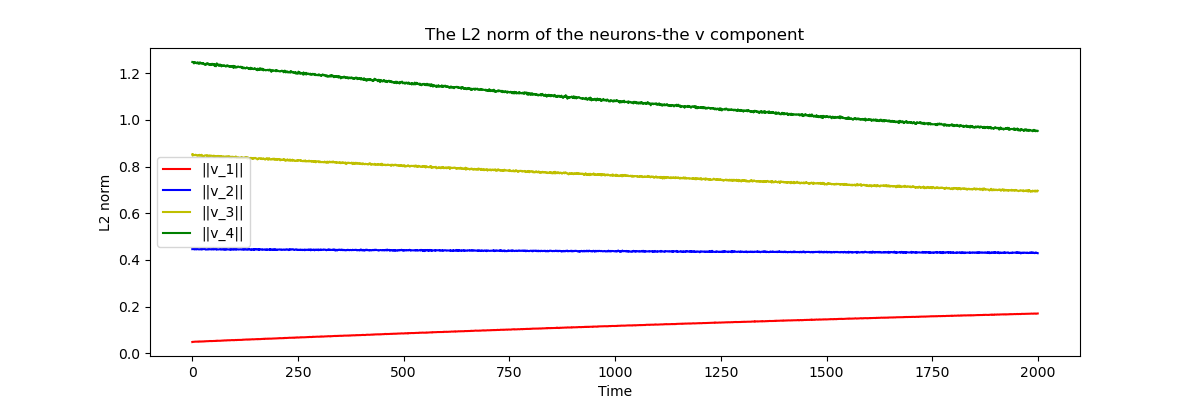} }}%
	\caption{The $L^2$ norm of the neurons component  $v_i$ after 666 iterations (upper figure) and after 2000 iterations (lower figure)}%
	\label{figHR2}%
\end{figure}

\begin{figure}%
	\centering
	\subfloat{{\includegraphics[width=12cm]{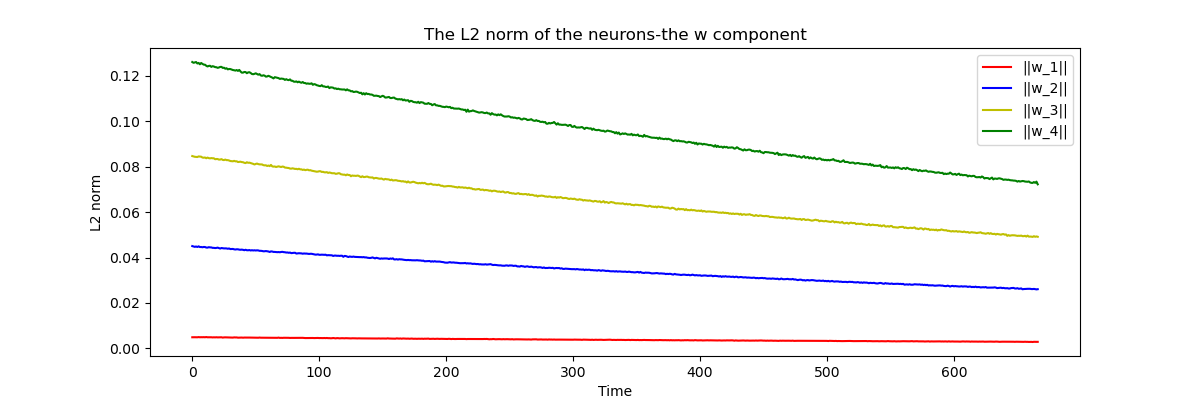} }}%
	\qquad
	\subfloat{{\includegraphics[width=12cm]{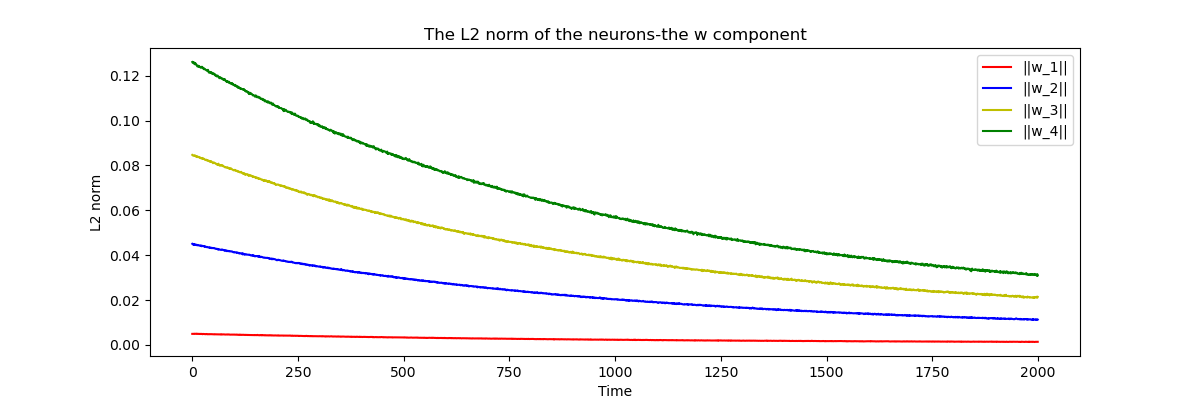} }}%
	\caption{The $L^2$ norm of the neurons component  $w_i$ after 666 iterations (upper figure) and after 2000 iterations (lower figure)}%
	\label{figHR3}%
\end{figure}

\begin{figure}%
	\centering
	\subfloat{{\includegraphics[width=12cm]{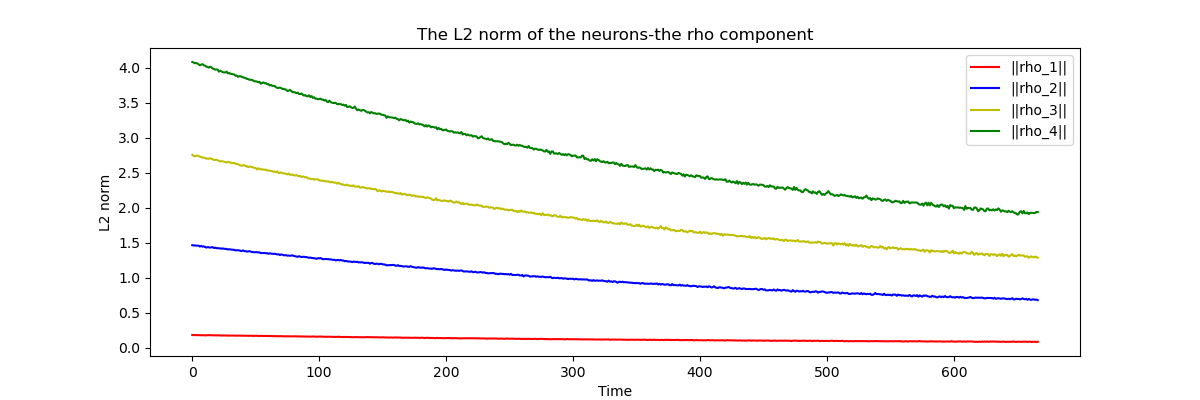} }}%
	\qquad
	\subfloat{{\includegraphics[width=12cm]{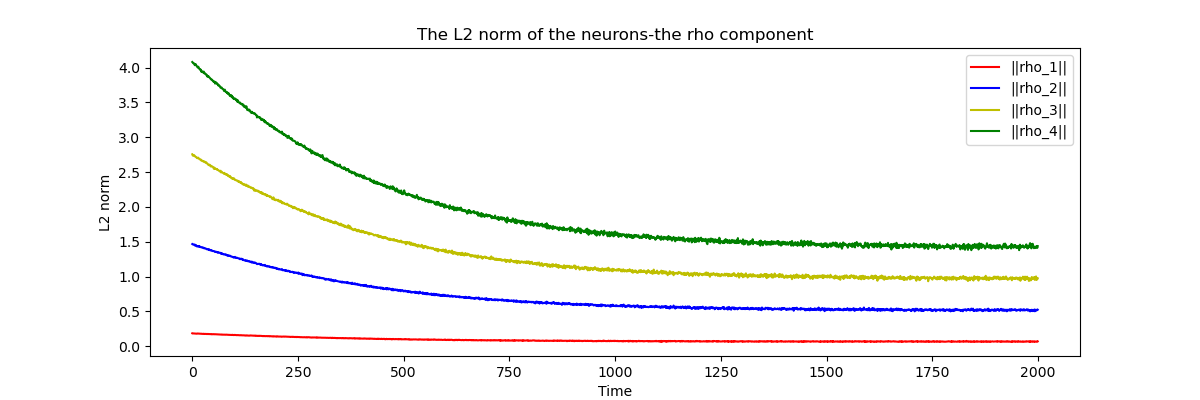} }}%
	\caption{The $L^2$ norm of the neurons component  $\rho_i$ after 666 iterations (upper figure) and after 2000 iterations (lower figure)}%
	\label{figHR4}%
\end{figure}

\begin{figure}%
	\centering
	\subfloat{{\includegraphics[width=12cm]{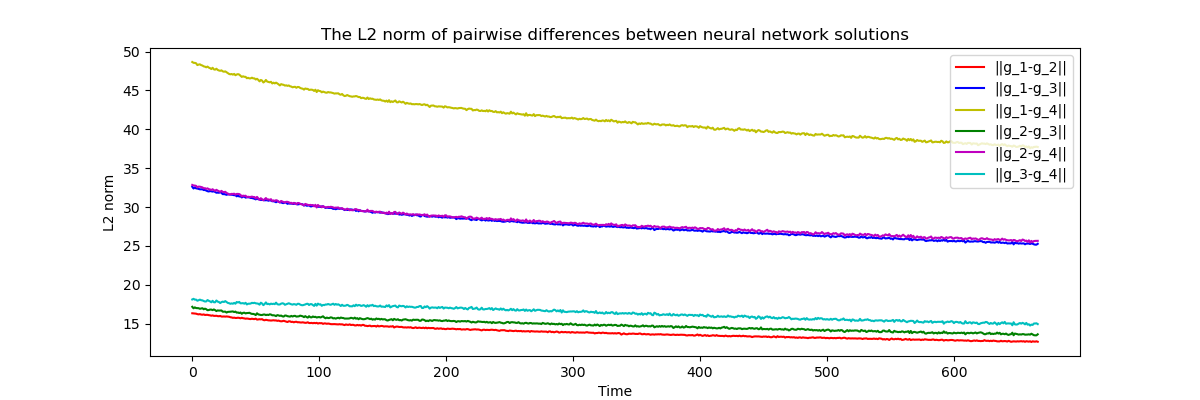} }}%
	\qquad
	\subfloat{{\includegraphics[width=12cm]{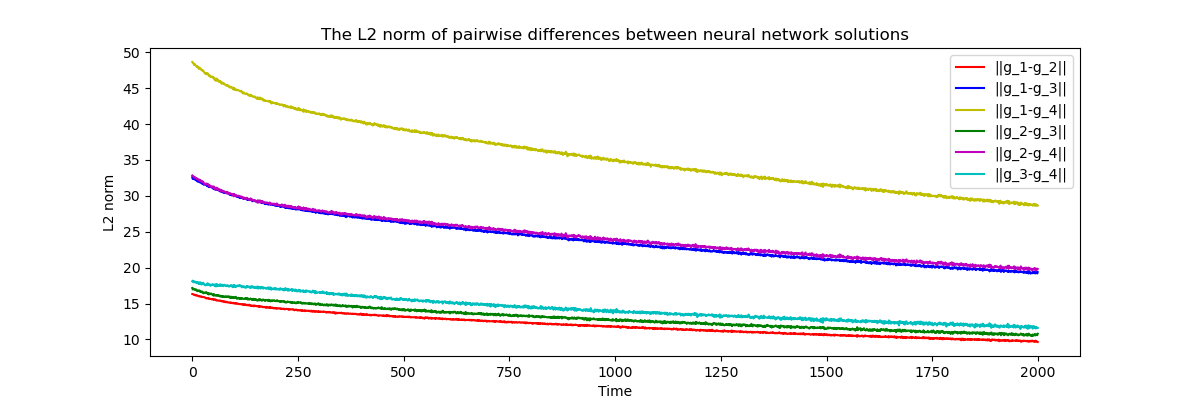} }}%
	\caption{The $L^2$ norm of pairwise differences between neural network solutions after 666 iterations (upper figure) and after 2000 iterations (lower figure)}%
	\label{figHR5}%
\end{figure}

\begin{table}[ht!]
	\begin{center}
		\caption{Comparison of the $u_i$ at the point $x=10,\ y=10$ }
		\label{tab:tableHR1}
		\begin{tabular}{l|c|c|r} 
			& Initial Value & At the 200 time step & At the 2000 time step\\
			\hline
			$u_1$ & 0.0.0075524   & 0.00160041   & 0.00441385   \\
			$u_2$ & 0.038242690   & 0.01277219   & 0.00990575    \\
			$u_3$ & 0.064522400   & 0.02294659   & 0.01506631    \\
			$u_4$ & 0.098338950   & 0.03892756   & 0.02388300      \\
		\end{tabular}
	\end{center}
\end{table}

\begin{table}[ht!]
	\begin{center}
		\caption{Comparison of the $v_i$ at the point $x=10,\ y=10$ }
		\label{tab:tableHR2}
		\begin{tabular}{l|c|c|r} 
			& Initial Value & At the 200 time step & At the 2000 time step\\
			\hline
			$v_1$ & 0.01623901   & 0.03495750   & 0.16725151  \\
			$v_2$ & 0.35418338   & 0.35641666   & 0.37220956\\
			$v_3$ & 0.61091695   & 0.60062472   & 0.52791116\\
			$v_4$ & 1.30404545   & 1.25993808   & 0.94827405\\
		\end{tabular}
	\end{center}
\end{table}

\begin{table}[ht!]
	\begin{center}
		\caption{Comparison of the $w_i$ at the point $x=10,\ y=10$ }
		\label{tab:tableHR3}
		\begin{tabular}{l|c|c|r} 
			& Initial Value & At the 200 time step & At the 2000 time step\\
			\hline
			$w_1$ & 0.00289204    & 0.00239707      & 0.00053399    \\
			$w_2$ & 0.05901589    & 0.04850523      & 0.00841518\\
			$w_3$ & 0.08612489    &  0.07083853     & 0.01235578\\
			$w_4$ & 0.14579938    & 0.11989174      & 0.02088769\\
		\end{tabular}
	\end{center}
\end{table}

\begin{table}[ht!]
	\begin{center}
		\caption{Comparison of the $\rho_i$ at the point $x=10,\ y=10$ }
		\label{tab:tableHR4}
		\begin{tabular}{l|c|c|r} 
			& Initial Value & At the 200 time step & At the 2000 time step\\
			\hline
			$\rho_1$ & 0.00516599     & 0.00377392   & 0.00065141   \\
			$\rho_2$ & 0.05720875     & 0.04118323   & 0.00309032 \\
			$\rho_3$ & 0.09721523     & 0.07001327   & 0.00511662\\
			$\rho_4$ & 0.13192033     & 0.09538789   & 0.00755887 \\
		\end{tabular}
	\end{center}
\end{table}
The synchronization result rigorously proved in this work is illustrated by the example with sample selections of the system parameters and a randomized set of initial data. Our numerical simulation also exhibits that the neuron potentials $u_i$ seem to be synchronized fastest within a limited time partly due to the memristor-potential coupling, while it takes much longer time to observe the synchronization on the other three variables $v_i$, $w_i$ and $\rho_i$.

\subsection{\textbf{Diffusive FitzHugh-Nagumo Equations with Memristor}}

Consider a model of memristive neural networks described by the diffusive FitzHugh-Nagumo equations cf. \cite{FH, LSY, Y3} with nonlinear weak coupling:
\beq \bl{FHN}
\begin{split}
	\frac{\pdr u_i}{\pdr t} & = \eta_2 \gd u_i + \ap_2 u_i (u_i - \gb_2)(1 - u_i) - \ga_2\, w_i - k_2 \tanh (\rho_i) u_i - P u_i \sum_{j=1}^m \G (u_j),  \\
	\frac{\pdr w_i}{\pdr t} & =  a_2 u_i + c_2 - b_2 w_i,    \\[2pt]
	\frac{\pdr \rho_i}{\pdr t} & = q_2 u_i - r_2 \rho_i,  
\end{split} 
\eeq
for $t > 0,\, x \in \gw \subset \mathbb{R}^{n}$ ($n \leq 3$), where $1 \leq i \leq m$ and $\gw$ is a bounded domain with locally Lipschitz continuous boundary $\partial \gw$. All the involved parameters are positive constants. The nonlinear function $\Gamma (s)$ is the same as in \eqref{Ga}.

In this system, the fast excitatory variable $u_i(t,x)$ refers to the transmembrane electrical potential of a neuron cell and the slow recovering variable $w_i(t, x)$ represents the integrated ionic current across the neuron membrane. The memductance $\rho_i (t, x)$ of the memristor caused by the electromagnetic induction flux across the neuron membrane is a scalar function. We impose the homogeneous Neumann boundary condition is $\frac{\pdr u_i}{\pdr \nu} (t, x) = 0,\, t > 0, \, x \in \partial \gw, 1 \leq i \leq m$, and the initial states of the system are denoted by 
$$
	u_i^0 (x) = u_i(0, x), \quad w_i^0 (x) = w_i (0, x), \quad \rho_i^0 = \rho_i (0, x), \quad 1 \leq i \leq m.
$$

As an application of the synchronization result Theorem \ref{ThM}, here we just check all the Assumptions in \eqref{Asp1} and \eqref{Asp2} are satisfied by this model of memristive FitzHugh-Nagumo neural networks \eqref{FHN}. 

In this case the generic functions in the Assumptions \eqref{Asp1} and \eqref{Asp2} are 
\beq \bl{fh}
\begin{split}
	f(s, \gs) &= \ap_2 s(s - \gb_2)(1 - s) - \ga_2\, \gs,  \\[3pt]
	\gL &= - b_2, \\[3pt]
	h(s, \gs) &= a_2s + c_2, \quad (s, \gs) \in \mathbb{R}^2.
\end{split}
\eeq
Check the Assumptions \eqref{Asp1} and \eqref{Asp2}: We can verify that 
\beq \bl{Vf1}
\begin{split}
	f(s,\gs)s &= -\ap_2 s^4 + \ap_2(1 + \gb_2)s^3 - \ap_2 \gb_2 s^2 - \ga_2\, s\,\gs  \\
	&\leq - \ap_2 \left(s^4 - \frac{3}{4} s^4 - \frac{1}{4}(1 + \gb_2)^4\right) + \ga_2 |s| |\gs|   \\
	&= - \frac{1}{4}\ap_2 s^4 + \ga_2 |s| |\gs| + \frac{1}{4}\ap_2 (1 +\gb_2)^4,   \quad (s, \gs) \in \mathbb{R}^2,
\end{split}
\eeq
and

\beq \bl{Vf2}
\begin{split}
	&\max \left\{ \frac{\partial f}{\partial s}(s, \gs), \left|\frac{\partial f}{\partial \gs}(s, \gs)\right| \right\} = \max\, \left\{- 3\ap_2 s^2 + 2\ap_2 (1 + \gb_2)s - \ap_2\gb_2,\, \ga_2\right\}  \\[3pt]
	\leq &\,\max\, \left\{- 3\ap_2 s^2 + \ap_2 s^2 + (1 + \gb_2)^2 - \ap_2\gb_2,\, \ga_2\right\} < \max\, \left\{ (1 + \gb_2)^2, \, \ga_2\right\}. 
\end{split}
\eeq	
Therefore the Assumption \eqref{Asp1} is satisfied. Moreover,
\beq \bl{Vf3}
\begin{split}
	\langle \gL \gs,\, \gs \rangle &= - b_2 |\gs |^2,   \\
	h(s, \gs)\gs &= (a_2 s + c_2)\gs \leq \frac{1}{4} a_2 |s|^2 |\gs| + (a_2 + c_2)|\gs |,  \\
	\frac{\partial h}{\partial s}(s, \gs) &= a_2 , \quad \frac{\partial h}{\partial \gs}(s, \gs) = 0,
\end{split}
\eeq
for $(s, \gs) \in \mathbb{R}^2$. Therefore Assumption \eqref{Asp2} is also satisfied. We record the specific parameters in \eqref{Asp1} and \eqref{Asp2} for this memristive FitzHugh-Nagumo neural network model as follows:
\beq \bl{npm}
\begin{split}
	\ap =  \frac{1}{4}\ap_2,  \quad  &\gl = \ga_2, \quad J =  \frac{1}{4}\ap_2 (1 +\gb_2)^4, \quad  \gb =  \max\, \left\{ (1 + \gb_2)^2, \, \ga_2\right\},  \\
	&\ga = b_2, \quad   q = \frac{1}{4}a_2, \quad L = a_2 + c_2, \quad \xi = a_2.
\end{split}
\eeq

Apply the proved synchronization Theorem \ref{ThM} to this memristive diffusive FitzHugh-Nagumo neural network model. We also reach the following result.
\begin{theorem} \bl{TFHN}
	For memristive diffusive FitzHugh-Nagumo neural networks with the model \eqref{FHN}, if the threshold condition \eqref{SC} with the parameters in \eqref{npm} is satisfied by the coupling strength coefficient $P$, then the neural network is exponentially synchronized in the state space $E = [L^2 (\gw, \mathbb{R}^3)]^m$ at a uniform exponential convergence rate $\delta (P)$ shown in \eqref{rate} with the parameters given in \eqref{npm}.
\end{theorem}

Now we numerically solve the differential equations of the memristive FitzHugh--Nagumo neural network model \eqref{FHN} in a two-dimensional square domain. We use the finite difference method for the numerical scheme and programmed in Python.

Choose the following parameters in the model \eqref{FHN}:
\begin{gather*}
	m = 4,\ \eta_2=10,\ \alpha_2=0.5, \beta_2=0.1 \gamma_2=0.05,\ k_2=0.1,   \\
	a_2=0.3,\ b_2=3,\ c_2=1,\ q_2=0.2,\ r_2=10, \\
	V=0.5, \ \ r=0.1.
\end{gather*}
Take the time-step to be 0.00025 and spatial-step to be 1 on a $32*32$ membrane. We compute the $L^2$ norm of the neuron membrane potential $u_i$, the recovering variable $w_i$, the memductance $\rho_i$, and also the vector solution $g_i$ of the model equations \eqref{FHN} in the energy space $E$. The plotted curves are shown in Figure \ref{fig1} to Figure \ref{fig4}. 

In comparison between the results after 333 iterations and after 1000 iterations in Figure \ref{fig1} to Figure \ref{fig3}, one can observe the synchronization tendency of the three characterizing variables $(u_i, w_i, \rho_i)$ among the neurons in the simulated mimristive FitzHugh-Nagumo neural network. From Figure \ref{fig4}, we observe that the pairwise differences of the $L^2$-norm $\|g_i - g_j\|$ tend to $0$. 

We can calculate the following constants involved in Theorem \ref{ThM} based on our selection of parameters, rounding up to 2 digits.
\begin{gather*}
	C_1 = 8.01, \quad C_2 = 2.89, \quad \mu = 0.25, \quad K = 94714.73, \quad  Q = 15101.69, \quad \\
	G = 9.67, \quad C^* = 0.4, \quad  \kappa = 15.49, \quad  P = 19.58, \quad \delta = 3.
\end{gather*}
The constant $C^*$ from Gargliardo-Nirenberg inequality is chosen to be $0.4$ based on \cite{BV}.

Table \ref{tab:table1} to Table \ref{tab:table3} list the sampled values of the three components $u_i, w_i$, and $\rho_i$ of the simulated solution $g_i$ at one same point in the domain at $t=0$ , at the 100 and 10000 time-step. It is seen that with a big difference on the initial values, after a certain time, the curves of $u_i$, $w_i$, and $\rho_i$ tend to be close to each other among various neurons. 
\begin{figure}%
	\centering
	\subfloat{{\includegraphics[width=12cm]{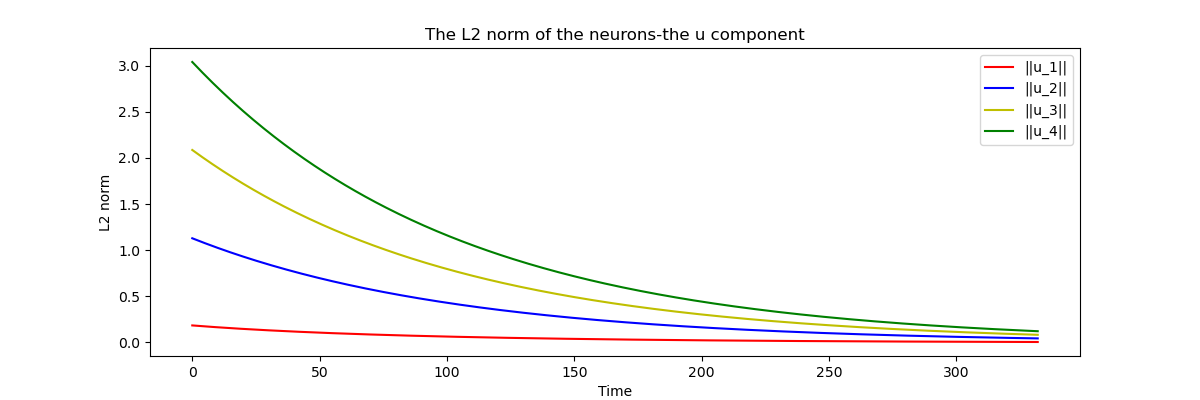}} }%
	\qquad
	\subfloat{{\includegraphics[width=12cm]{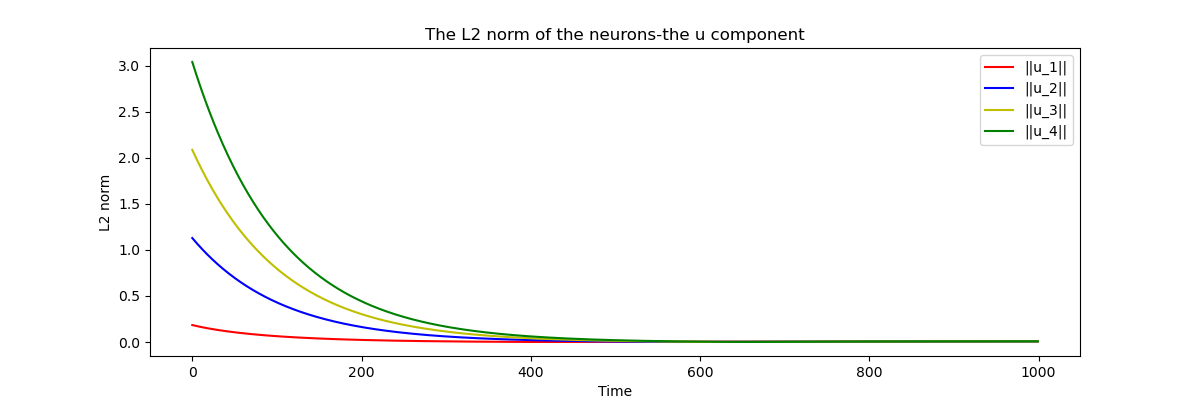} }}%
	\caption{The $L^2$ norm of the neurons component $u_i$ after 333 iterations (upper figure) and after 1000 iterations (lower figure)}%
	\label{fig1}%
\end{figure}

\begin{figure}%
	\centering
	\subfloat{{\includegraphics[width=12cm]{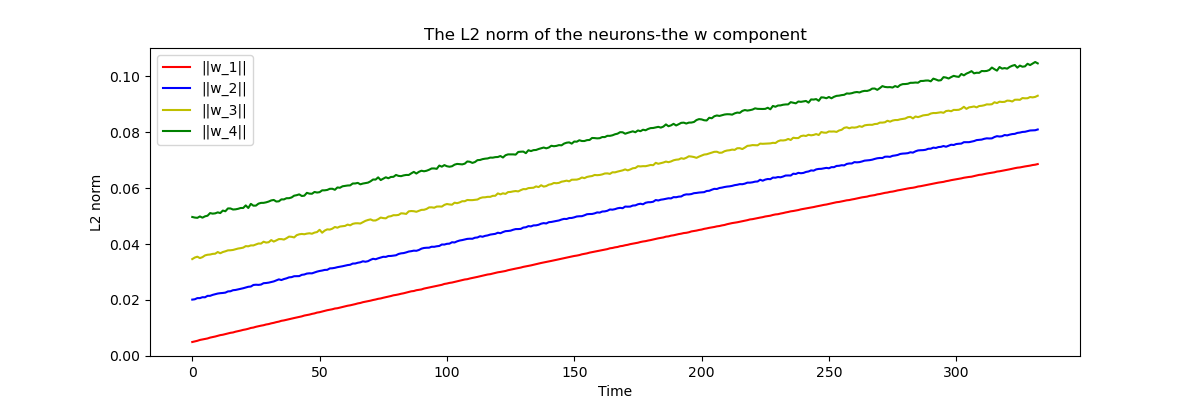} }}%
	\qquad
	\subfloat{{\includegraphics[width=12cm]{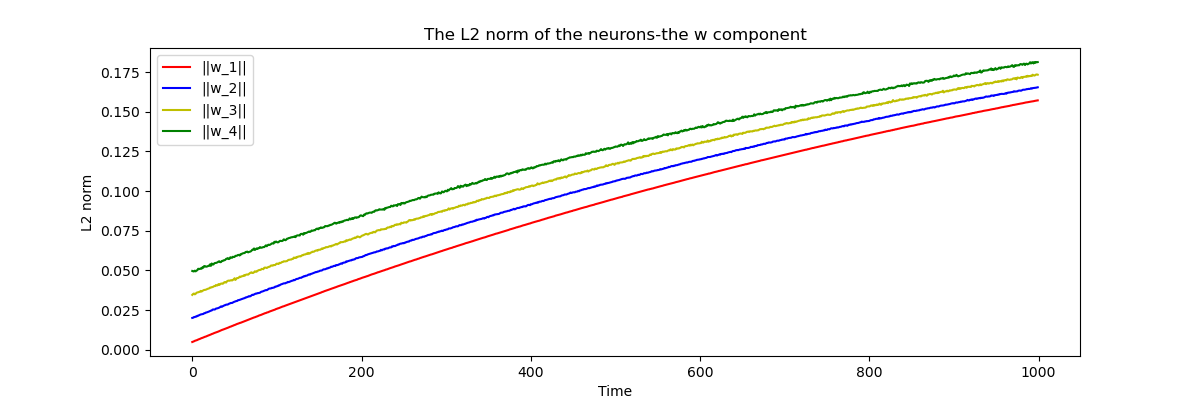} }}%
	\caption{The $L^2$ norm of the neurons component  $w_i$ after 333 iterations  (upper figure) and after 1000 iterations (lower figure)}%
	\label{fig2}%
\end{figure}

\begin{figure}%
	\centering
	\subfloat{{\includegraphics[width=12cm]{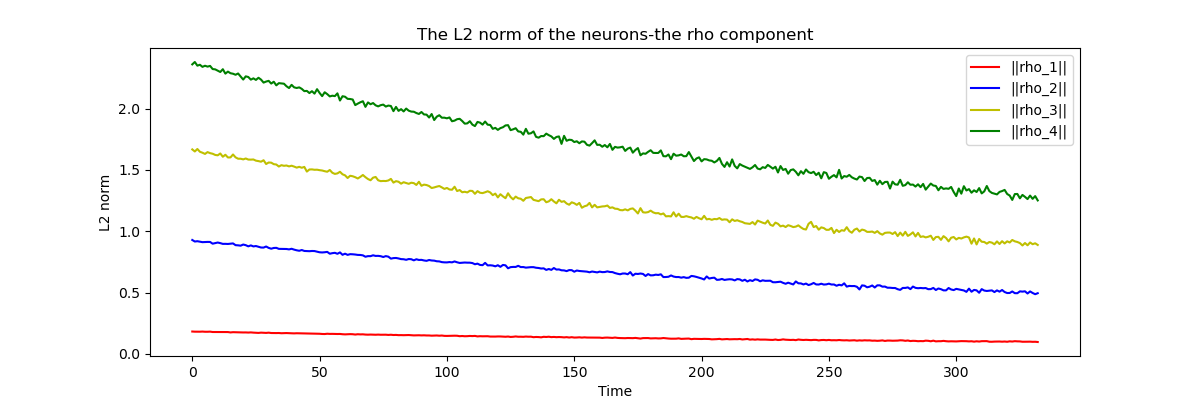} }}%
	\qquad
	\subfloat{{\includegraphics[width=12cm]{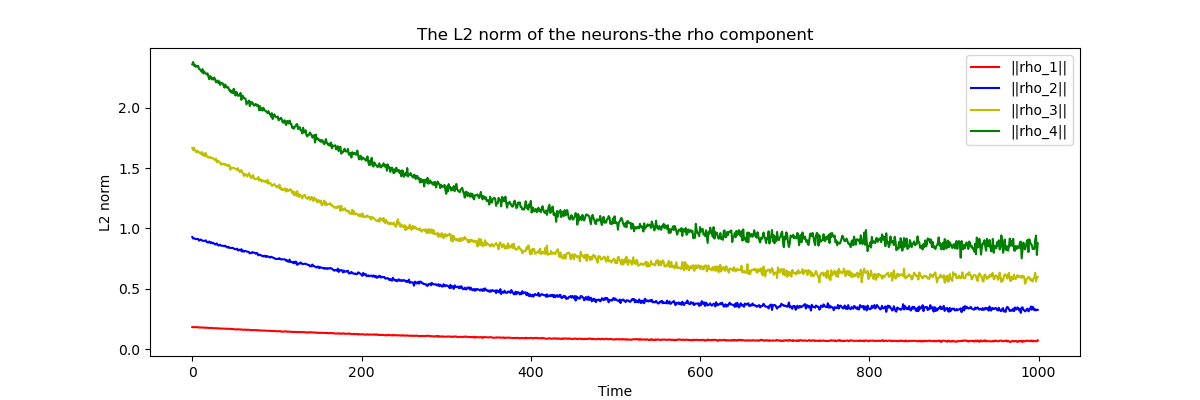} }}%
	\caption{The $L^2$ norm of the neurons component  $\rho_i$ after 333 iterations (upper figure) and after 1000 iterations (lower figure)}%
	\label{fig3}%
\end{figure}

\begin{figure}%
	\centering
	\subfloat{{\includegraphics[width=12cm]{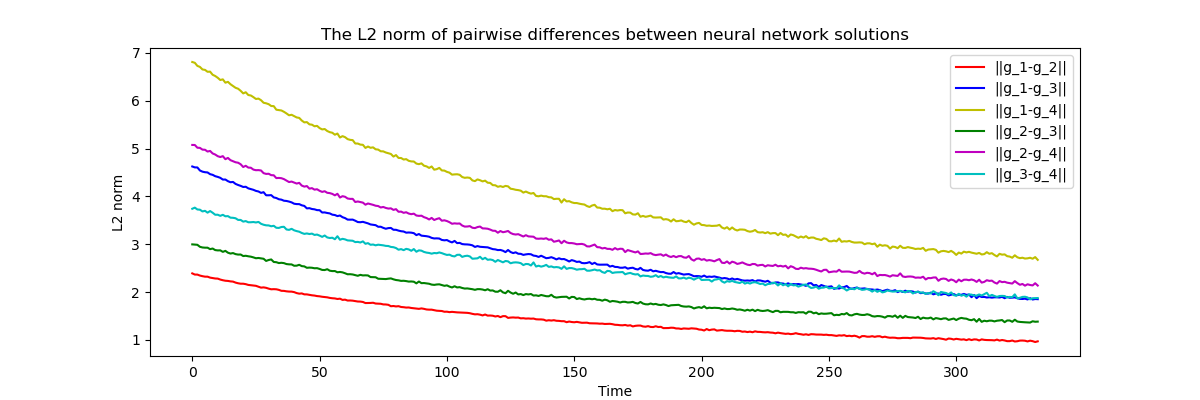} }}%
	\qquad
	\subfloat{{\includegraphics[width=12cm]{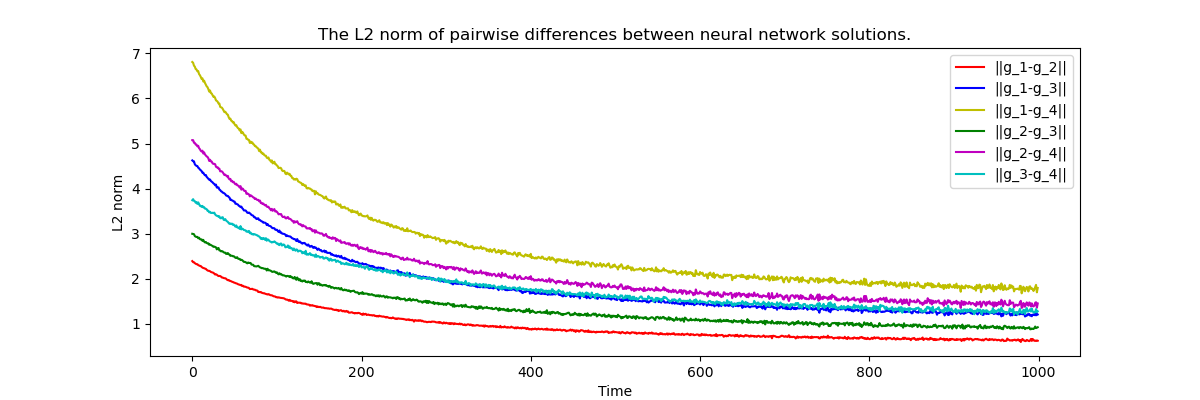} }}%
	\caption{The $L^2$ norm of pairwise differences between neural network solutions after 333 iterations (upper figure) and after 1000 iterations (lower figure)}%
	\label{fig4}%
\end{figure}

\begin{table}[ht!]
	\begin{center}
		\caption{Comparison of the $u_i$ at the point $x=10,\ y=10$ }
		\label{tab:table1}
		\begin{tabular}{l|c|c|r} 
			& Initial Value & At the $100$ time step & At the $1000$ time step\\
			\hline
			$u_1$ & 0.00734587  &  0.00229806   & -0.00021516  \\
			$u_2$ & 0.03147844  &  0.01288047   & -0.00022583 \\
			$u_3$ & 0.06128196  &  0.02419445   & -0.00022394\\
			$u_4$ & 0.09076842  &  0.03540031   & -0.00026045\\
		\end{tabular}
	\end{center}
\end{table}

\begin{table}[ht!]
	\begin{center}
		\caption{Comparison of the $w_i$ at the point $x=10,\ y=10$ }
		\label{tab:table2}
		\begin{tabular}{l|c|c|r} 
			& Initial Value & At the $100$ time step & At the $1000$ time step\\
			\hline
			$w_1$ & 0.00280643    & 0.02672859      & 0.17726598   \\
			$w_2$ & 0.02656693    & 0.04889201      & 0.18859465\\
			$w_3$ & 0.01219399    & 0.03569609      & 0.18192625\\
			$w_4$ & 0.09637967    & 0.11393263      & 0.22179753\\
		\end{tabular}
	\end{center}
\end{table}

\begin{table}[ht!]
	\begin{center}
		\caption{Comparison of the $\rho_i$ at the point $x=10,\ y=10$ }
		\label{tab:table3}
		\begin{tabular}{l|c|c|r} 
			& Initial Value & At the $100$ time step & At the $1000$ time step\\
			\hline
			$\rho_1$ & 0.00822771    & 0.02672859   & 0.00067401  \\
			$\rho_2$ & 0.01216352    & 0.04889201   & 0.00101167\\
			$\rho_3$ & 0.08205714    &  0.03569609  & 0.00674807 \\
			$\rho_4$ & 0.08914909    & 0.11393263   & 0.00734463\\
		\end{tabular}
	\end{center}
\end{table}
The synchronization result rigorously proved in this work is illustrated by the presented example with sample selections of the system parameters and a randomized set of initial data. Our numerical simulation also exhibits that the neuron potentials $u_i$ seem to be synchronized fastest within a limited time, while it takes much longer time to observe the synchronization on the other two variables $w_i$ and $\rho_i$. 

This observation actually enhances the neurodynamical conjecture that adding a nonlinear memristor coupling in the neuron potential equation would accelerate the synchronization for the main variable of neuron membrane potential. On the other hand, it also hints that although the main Theorem~\ref{ThM} confirmed the exponential synchronization has a uniform but maybe small convergence rate, each of the three components may have a different synchronization rate, which turns out to be a new interesting problem for further research.

\textbf{Conclusions}. 
We summarize the new contributions of the results in this paper.

\vspace{3pt}
1. In this paper we propose and study a general mathematical framework that can cover many typical and useful partial-ordinary differential equation models to characterize spatiotemporal dynamics of biological neural networks with memristors and weak synaptic coupling, which is a challenging and open problem in mathematical neuroscience and potentially in complex artificial learning dynamics.  

2. The advancing contributions of this work are in three aspects. 

First it is proved in Section 3 and Section 4 that the solution semiflow of these memristive neural networks exhibits dissipative dynamics in common and admits ultimately uniform bounds in multiple norms. 

Second and more important is the exponential synchronization Theorem \ref{ThM} and Theorem \ref{TM}, in which we rigorously proved an explicit threshold condition in terms of the involved biological parameters and one mathematical parameter to ensure a synchronization at a uniform exponential rate in the $L^2$ energy norm. 

Third we provide an effective analytic approach and a significant methodology to pursue the synchronization investigation through scaled \emph{a priori} estimates, leverage of dynamic integral inequalities, and sharp interpolation inequalities (such as the crucial Gagliardo-Nirenberg inequalities on Sobolev spaces) to tackle and control the memristive effect and nonlinearity by the weak synaptic coupling strength only. 

3. Two illustrative applications of the main result on synchronization are presented by the memristive diffusive Hindmarsh-Rose neural networks and FitzHugh-Nagumo neural networks. 

It is expected that the mathematical framework and approach presented in this work and related computational simulations can be further generalized to a broader field and integrated with more applications in neurodynamics and network dynamics.

\bibliographystyle{amsplain}

\end{document}